\RequirePackage{tikz}

\documentclass[pdflatex,sn-mathphys]{sn-jnl}% Math and Physical Sciences Reference Style
\usepackage{tikz}
\usepackage{pgfplots}
\usepackage{bbm}
\usepackage{enumerate}

\usepackage{makecell}

\usepackage{subcaption}
\jyear{2021}%

\theoremstyle{thmstyleone}%
\newtheorem{theorem}{Theorem}%  meant for continuous numbers
\newtheorem{proposition}[theorem]{Proposition}% 
\newtheorem{lemma}[theorem]{Lemma}% 

\theoremstyle{thmstyletwo}%
\newtheorem{example}{Example}%
\newtheorem{remark}{Remark}%

\theoremstyle{thmstylethree}%

\newtheorem{assumption}{Assumption}%

\begin{document}

\title[Single-Server Queue Control]{Arrival-Intensity Control for A Single-Server Queue in Heavy Traffic}


\author*[1]{\fnm{Bowen} \sur{Xie}}\email{bxie@uakron.edu}

\author[2]{\fnm{Haoyu} \sur{Yin}}\email{h.yin@wustl.edu}


\affil*[1]{\orgdiv{Department of Mathematics}, \orgname{College of Engineering and Polymer Science, The University of Akron}, \orgaddress{\street{290 East Buchtel Ave}, \city{Akron}, \postcode{44325-4002}, \state{OH}, \country{US}}}

\affil[2]{\orgdiv{Department of Electrical \& Systems Engineering}, \orgname{Washington University in St. Louis}, \orgaddress{\street{1 Brookings Drive},
\city{St. Louis}, \postcode{63130}, \state{MO}, \country{US}}}

% \affil[3]{\orgdiv{Department}, \orgname{Organization}, \orgaddress{\street{Street}, \city{City}, \postcode{610101}, \state{State}, \country{Country}}}


\abstract{We study a single-server queue control problem (QCP) in a heavy-traffic regime, extending the framework from Lee and Weerasinghe (2011). The state process represents the offered waiting time. Service times and patience times form independent i.i.d. sequences with general distributions. 
We formulate an infinite-horizon discounted QCP that balances the cost of controlling the arrival intensity against a penalty for server idleness. 
A distinctive feature of the formulation is the nonstandard decreasing operational running cost arising from the arrival-intensity control mechanism. 
Under suitable heavy-traffic assumptions, the diffusion-scaled offered waiting-time process converges to a regulated diffusion, leading to an associated diffusion control problem (DCP). 
We find the optimal control of the associated DCP by incorporating the Legendre-Fenchel transform and a formal Hamilton-Jacobi-Bellman (HJB) equation. We then construct a sequence of intensity controls for the prelimit queueing systems from the DCP-optimal feedback and establish its asymptotic optimality within the specified heavy-traffic admissible control class. 
Beyond theoretical analysis, numerical experiments further examine
whether reinforcement learning can
approximate the optimal policy with discounted costs close to the HJB
reference. Policies are trained from simulated state transitions
and realized costs, and are
evaluated against the independently computed HJB feedback.}

\keywords{Arrival-Intensity control, Waiting time process,
State-dependent intensity,
Legendre-Fenchel transform, Hamilton-Jacobi-Bellman equation, Reinforcement learning}


\pacs[MSC Classification]{(Primary) 60K25, 93E20, (Secondary) 90B22}

\maketitle

\section{Introduction}
\label{sec:sample1}

This paper concerns an arrival-intensity control problem for a single-server queue with impatient customers in a heavy-traffic regime. 
The state variable is the offered waiting time, defined as the amount of time that a hypothetical infinitely patient customer would wait before entering service. 
Customers arrive according to a state-dependent intensity, have generally distributed service times, and may abandon the system when their patience expires. 
Our objective is to regulate the arrival intensity so as to balance the cost of controlling incoming traffic against the cost associated with server idleness. 
Beyond theoretical analysis, we complement the theoretical results with numerical experiments and assess reinforcement learning (RL) as an alternative learning approach. This work bridges theoretical queueing models and data-driven control, offering insights into the strengths and limitations of both approaches.

\subsection{Motivation}

In practical systems such as telephone communication centers, internet bandwidth sharing, and cloud computing systems, customers often lack visibility into queue lengths and instead rely on provided waiting times. 
For instance, in systems where job sizes or service requirements are observable or can be accurately estimated, work-load information can be translated into an offered-waiting-time estimate. 
In our model, the state process represents the offered waiting time, i.e., the time a hypothetical customer with infinite patience would wait upon arrival.
Customer arrivals follow a state-dependent intensity, reflecting how system conditions influence incoming traffic. This formulation builds on the framework of \cite{lee2011convergence}, where arrival rates are dynamically adjusted based on system state. A practical interpretation of this intensity arises in settings where each customer brings a task with a deadline or observable workload, as discussed in \cite{foss2001optimality} and \cite{der2022scalable}. The system manager infers deadlines and service requirements to determine the offered waiting time and controls the arrival intensity through some intensity policies.
The primary objective is to identify an optimal arrival-intensity control policy that minimizes cost under some cost structure.

To formulate the single-server waiting time control problem (QCP), we consider a sequence of $n$ single-server queueing systems under the heavy-traffic approximation framework of \cite{lee2011convergence}. In the $n$th system, arrivals occur at a state-dependent rate of order $O(n)$, while service times are i.i.d., generally distributed with a mean of order $O(1/n)$ for large $n$. Patience times are also i.i.d., potentially depending on $n$. As $n \to \infty$, the system approaches heavy traffic, balancing high arrival rates with short service times.
Under mild assumptions, \cite{lee2011convergence} showed that the diffusion-scaled offered waiting time converges to a stochastic differential equation. Building on this, we formulate the QCP using the weak convergence result (Proposition \ref{diffusion limit}) and define a cost functional \eqref{cost functional of QCP} that includes a operational running cost for regulating arrivals and a penalty for idleness. The heavy-traffic limit of the QCP leads to a diffusion control problem (DCP), as given in \eqref{costfunctional}.
To solve the control problem, we establish
% apply the ``BIGSTEP" method (cf. \cite{harrison1996bigstep, harrison1998heavy, bell2001dynamic}) in the context of waiting time processes, using 
discrete-view control policies by constructing near-optimal solutions. 
We show that an asymptotically optimal policy for the QCP can be derived from the optimal solution to the DCP, which is typically more tractable.
Leveraging the Legendre-Fenchel transform, we derive a smooth solution to the associated HJB equation, which involves solving a free-boundary type problem defined by a system of nonlinear ODEs. 
Finally, the resulting feedback control serves as a nearly optimal policy for the QCP under heavy traffic, yielding an asymptotically optimal solution to the original control problem.

\subsection{Contributions}

The main contributions of this paper are summarized as follows.

First, we formulate an arrival-intensity control problem directly for the offered waiting-time process of a single-server queue with impatient customers. 
The control is introduced through a state-dependent perturbation of the arrival intensity at the heavy-traffic scale. Unlike a conventional control-effort variable, the control here represents the reduction in arrival intensity relative to its nominal heavy-traffic level. 
Thus, a larger control corresponds to a lower arrival intensity and consequently a lower operating cost, which motivates the decreasing operational running cost considered in this paper. Conversely, reducing the arrival intensity increases the likelihood of server idleness, generating a competing idleness cost. This formulation therefore produces a nonstandard trade-off between the decreasing running cost associated with arrival regulation and the cost generated by server idleness. In particular, the formulation permits both positive and negative diffusion-scale controls in the prelimit system, corresponding to arrival intensities below and above the nominal heavy-traffic rate, respectively. 

Second, we characterize and solve the limiting diffusion control problem. The associated HJB equation contains a nonlinear optimization term arising from the arrival-intensity control. Using the Legendre-Fenchel transform of the operational running cost, we reduce the HJB equation to a nonlinear ordinary differential equation and establish the existence of a critical solution satisfying the required boundary and asymptotic conditions.  A verification argument then identifies this solution with the value function and yields the optimal state-feedback control.

Then, we establish the connection between the DCP-optimal feedback and the original queueing control problem. From the optimal diffusion feedback $u^*$, we construct a sequence of prelimit controls $\{u_n^*\}$ satisfying the conditions required by the heavy-traffic framework. We prove an asymptotic lower bound for admissible competing control sequences and show that the DCP-derived sequence attains the limiting value. Consequently, $\{u_n^*\}$ is asymptotically optimal for the sequence of queueing control problems. This formulation departs from existing queue optimization literature, particularly in its treatment of control applied to the waiting time process and the structure of operational cost function. The resulting control problem presents unique challenges due to the non-standard control mechanism, non-polynomial cost growth, and the drift adjustment induced by varying arrival rates. These features also affect the convergence of the cost functional when translating asymptotically optimal policies from the DCP to the original QCP.

Finally, we provide numerical experiments supporting the analytical results by taking some examples of system primitives. 
We implement the critical-parameter construction used in the HJB analysis to numerically compute the value function and optimal feedback control. 
Additionally, inspired by recent advances in learning for queueing systems
\cite{walton2021learning,dai2022queueing,jia2024online}, we further
examine whether reinforcement learning (RL) can approximate the optimal
feedback policy and achieve discounted costs close to those of the
HJB reference. The HJB approach uses the model coefficients to solve
a nonlinear differential equation, while sample-based learning
constructs policies from state transitions and realized costs.
To assess this approach for the proposed waiting-time control problem,
we formulate a time discretization of the DCP as a discounted Markov
decision process (MDP) and apply Q-learning, proximal policy
optimization (PPO), and soft actor--critic (SAC), together with the
cross-entropy method (CEM). The policies are trained from simulated data, and the
independently computed HJB feedback is reserved for evaluation.
Their discounted costs are compared under the same simulation
conditions to assess how closely the learned policies approach the
performance of the HJB reference. This study also provides a
computational basis for exploring extensions in which the dynamics
are accessible through simulation and the associated HJB equation
is difficult to solve.

\subsection{Literature Review}

arrival-intensity control in stochastic systems dates back to \cite{bather1966continuous} and has since been widely studied across fields such as finance, operations research, and management science. It plays a central role in queueing, inventory control, and dynamic pricing, where state-dependent controls are commonly used.

In mathematical finance, stochastic control problems arise in models such as exchange rate target zones and central bank interventions (cf. 
\cite{krugman1991target,bertola1992target,miller1996optimal,cadenillas1999optimal}). 
In queueing and production-inventory systems under heavy traffic, buffer-length control problems often involve balancing the cost of rejected customers due to full buffers against the cost of abandonment from long waits, leading to cost minimization formulations (cf. 
\cite{koccauga2010admission,weerasinghe2013abandonment,weerasinghe2016optimal,biswas2017ergodic,arapostathis2018infinite,yang2020optimality,xie2024long,xie2024multi}).

In the queueing literature under heavy traffic, extensive work has addressed system optimization for both single- and many-server regimes (cf. \cite{bell2001dynamic,george2001dynamic,atar2004scheduling,ata2005heavy}). 
For instance, \cite{bell2001dynamic} studied a dynamic scheduling problem with two arrival streams and two heterogeneous servers, establishing the asymptotic optimality of a threshold policy. \cite{george2001dynamic} analyzed a single-server queue with Poisson arrivals and state-dependent service rates, focusing on holding cost minimization. \cite{atar2004scheduling} considered a many-server system with multiple customer classes and impatient behavior, addressing a discounted cost optimization problem. \cite{ata2005heavy} proposed a discrete review policy for open stochastic processing networks with feedback routing and overlapping server capabilities.
More recently, \cite{arapostathis2019optimal} examined multiclass many-server queues with arrival, service, and abandonment rates modulated by a common finite-state Markov process, solving both discounted and ergodic control problems. Additional contributions to control in queueing models include works by \cite{weerasinghe2007abelian,ward2008asymptotically,ghosh2010optimal,weerasinghe2013abandonment,weerasinghe2016optimal,weerasinghe2018controlling}.

In the inventory control literature, \cite{gross1971one} introduced one-for-one ordering models where lead times depend on the number of outstanding orders. \cite{ozer2004inventory} studied a periodic-review, stochastic, capacitated production system with advance demand information, establishing optimal state-dependent policies under various operational constraints. \cite{kutzner2013optimal} examined a single-item inventory problem with stochastic demand and state-dependent random yields, where order sizes follow an order-up-to policy.
State-dependent control problems also arise in dynamic pricing. \cite{dotsey2005implications} explored state-dependent pricing in macroeconomic models, showing its impact on economic dynamics. \cite{dube2008category} considered the category pricing problem with state-dependent utility and demonstrated that the presence of loyalty materially affects optimal pricing. 
\cite{dong2011state} proposed anticipatory pricing for real-time freeway management, leveraging predicted traffic conditions to improve pricing strategies.

Our single-server queue waiting time control problem differs significantly from existing models in its treatment of arrival intensity and arrival-intensity control, particularly due to the structure of the operational running cost. Rather than modeling the queue length, we focus on the offered waiting time as the state process. Since the arrival intensity depends on the current waiting time, the system manager may adjust the arrival rate by an order of $O(\sqrt{n})$, which directly influences the drift of the limiting diffusion process, effectively making the drift a control variable.
This leads to a waiting time control problem where arrival-intensity control is embedded in the arrival intensity. As a result, the operational running cost function exhibits a non-trivial structure, which in turn affects the convergence of the cost functional in heavy traffic. Such adjustments are known as “thin control” in controlled queueing networks (cf. \cite{ata2006dynamic,ward2008asymptotically,ghosh2010optimal}).

In recent years, reinforcement learning has gained traction as a data-driven approach to solving stochastic control problems. For instance, \cite{walton2021learning} provided insights into applying supervised, online, and reinforcement learning to queueing systems. \cite{dai2022queueing} extended the theoretical framework of advanced policy gradient methods to MDPs with infinite state spaces, unbounded costs, and long-run average objectives--features common in complex queueing networks. \cite{quer2022connecting} explored the connection between stochastic optimal control and RL, showing how optimal control problems can be framed within the RL paradigm.
Recent work has also focused on developing RL algorithms tailored to queueing control, often outperforming traditional heuristics across various traffic regimes. For instance, \cite{jia2024online} proposed two online learning algorithms--Batch Upper Confidence Bound and Batch Thompson Sampling (BTS)--for price-based revenue management with reusable resources. Additional contributions include \cite{dai2021refined,feng2021scalable,chen2023online,baron2024supervised}, which further demonstrate the growing role of RL in queueing control.

\subsection{Organization}

The remainder of the paper is organized as follows. Section \ref{Sec1 in CH4: Stochastic Model} introduces the basic model and formulates the QCP, along with preliminary results under heavy traffic. Section \ref{Sec3 in CH4: Diffusion Control Problem (DCP)} introduces the associated DCP derived from the heavy-traffic limit. In Section \ref{sec: analysis of the HJB equation}, we introduce the Legendre-Fenchel transform and analyze the formal HJB equation. Section \ref{Subsec2 of Sec3 in CH4: Optimal Solution of the DCP} establishes a verification lemma and derives a feedback control policy shown to be optimal for the DCP. Section \ref{Sec3 in CH4: Asymptotic Optimality} demonstrates how this strategy ``translates" into a nearly optimal policy for the original QCP. 
Section~\ref{sec:numerical} provides a numerical solution of the HJB equation and illustrates the resulting optimal feedback control. 
Additional numerical experiments, including sensitivity analysis, and policy comparisons, are provided in Appendix~\ref{appendix:additional-numerical}.
Section~\ref{sec:RL} presents the discrete-time control model, learning methods, and evaluation of learned policies against the HJB reference. 
Section \ref{sec conclusion} concludes the paper.

\textbf{Notation.} 
Let $\mathbb{N}$ represent the set of positive integers. Let $\mathbb{R}$ denote the set of real numbers and $\mathbb{R}_+$ denote the set of non-negative real numbers. For $0< T\leq \infty$, let $D[0, T]$ denote the Skorokhod space of c\`adl\`ag functions from $[0, T]$ to $\mathbb{R}$, equipped with the usual Skorokhod topology. 
The uniform norm on $[0, T]$ for a stochastic process $X$ in $D[0, T]$ is defined by $\lvert X\rvert_T = \sup_{t\in[0, T]}\lvert X(t)\rvert$. 
% \begin{equation*}
%     \|X\|_T = \sup_{t\in[0, T]}|X(t)|.  
% \end{equation*}
Throughout, we use $\Rightarrow$ to denote weak convergence in the Skorokhod space $D[0, T]$. For any $a, b\in\mathbb{R}$, $a^+ = \max\{a, 0\}$, $a^- = \max\{-a, 0\}$, $a\wedge b = \min\{a, b\}$, and $a\vee b = \max\{a, b\}$.

\section{Stochastic Model}
\label{Sec1 in CH4: Stochastic Model}

In this section, we introduce the single-server waiting-time model and formulate the associated queueing control problem. The underlying heavy-traffic framework for the offered waiting-time process follows \cite{lee2011convergence}. Our objective is to introduce an arrival-intensity control at the prelimit level and subsequently relate the resulting queueing control problem to a diffusion control problem. 
The optimal feedback policy obtained from the diffusion control problem is then used to construct a sequence of controls for the prelimit queueing systems and to establish their asymptotic optimality.

% % Reset equation number 
% \numberwithin{equation}{section}
% \renewcommand{\theequation}{4.\arabic{equation}}
% % reset the counter
% \setcounter{equation}{1}

\subsection{Basic Model}
\label{Subsec1 of Sec1 in CH4: Basic Model}

We consider a single-server queueing system with impatient customers on a probability space $(\Omega, \mathcal{F}, P)$. 
The state process is defined as the offered waiting time, namely, the amount of time that a hypothetical customer with infinite patience would wait before entering service if the customer arrived at the current time. 
For $j\geq1$, we define an ordered triple $(t_j, v_j, d_j)$ for $j\geq1$ to be a collection of the information related to the $j$th arriving customer, where $t_j$ denotes the arrival time, $v_j$ denotes the service time, and $d_j$ denotes the patience time of the $j$th customer. 
It is assumed that the service times $(v_j)$ and the patience times $(d_j)$ are i.i.d. and independent of each other. 
We further assume $\mathbb E(t_j)<\infty$, $\mathbb E(v_1)=1$, and $\text{Var}(v_1)=\sigma_s^2< \infty$.  
Moreover, the cumulative distribution function of $d_1$ is given by $F$. We define the filtration $(\hat{\mathcal{F}}_n)_{n\geq0}$ by $\hat{\mathcal{F}}_0:=\sigma(t_1)$ and
\begin{equation}
    \hat{\mathcal{F}}_n:=\sigma((t_1, v_1, d_1), \cdots, (t_n, v_n, d_n), t_{n+1})\subseteq\mathcal{F}, \text{ for } n\geq 1. 
    \label{hat F_n filtration}
\end{equation}
% for $n\geq1$. 

Let $A(t)$ denote the number of arrivals by time $t$. If the initial workload is $x\geq 0$, then the offered waiting time process is given by: 
\begin{equation}
    V(t) := x + \sum_{j=1}^{A(t)} v_j\mathbbm{1}_{[V(t_j-)<d_j]}-\int_0^t\mathbbm{1}_{[V(s)>0]}ds, 
    \label{offered waiting time process}
\end{equation}
for $t\geq0$, where $x\geq0$ is some fixed constant that provides the amount of time needed to empty the whole system if no arrival occurs at time zero, and the last integral term represents the cumulative server busy time up to time $t$. 

It is straightforward that this quantity depends on the service time of those non-abandoning customers who are already ahead in the queue waiting for service. 
It reflects the amount of workload needed to empty the queue provided no new arrivals after time $t$. 
Additionally, an arriving customer contributes its service requirement to the workload only when its patience time exceeds the offered waiting time observed immediately before arrival. Between effective arrivals, the process decrease at unit rate whenever it is positive. Consequently, $V(\cdot)$ is a nonnegative and has c\`adl\`ag sample paths with upward jumps corresponding to customers who remain in the system, and $V(t_n-)$ is $\hat{\mathcal{F}}_{n-1}$ measurable. 
% More precisely, a upward jump of size $v_j$ comes from the non-abandoning customer $j$, and it is continuous non-increasing and satisfies $V(t)=(V(t_j)-(t-t_j))^+$ for all $t\in[t_j, t_{j+1})$. Observe that if the first arrival occurs strictly after time $0$, i.e., $t_1>0$, then we have $V(0) = x$ for some constant $x\geq0$ and $V(t)=(x-t)^+$ for all $t\in[0, t_1)$. 
% If the first non-abandoning arrival occurs exactly at time 0, i.e., $t_1 = 0$, then a hypothetical customer arriving right after time 0 would have to wait for the amount of time the first customer waits in addition to its service time, which can also be denoted by some $x \geq0$. Note that this is a different quantity compared with the previous case; however, to reduce complexity, we may include both situations using a single initial workload $x$. 
% Figure~\ref{fig: sample path of the offered waiting time V(t)} provides a sample visual illustration of these descending and jumping properties. 
It is noteworthy that since $V(\cdot)$ is only affected by those non-abandoning customers, there are no downward jumps due to the abandonment. 
% \begin{figure}[htph]
%     \centering
%     %\includegraphics{}
%     \begin{tikzpicture}
%     \draw[->] (-0.5, 0) -- (7.5, 0) node[right] {$t$};
%     \draw[->] (0, -0.5) -- (0, 3.5) node[above] {$V(t)$};
%     % Closed dots
%     \node at (0, 1.2) [scale=0.7, circle,fill,inner sep=1.5pt]{};
%     \node at (1.8, 2.2) [scale=0.7, circle,fill,inner sep=1.5pt]{};
%     \node at (3, 2) [scale=0.7, circle,fill,inner sep=1.5pt]{};
%     \node at (6, 3) [scale=0.7, circle,fill,inner sep=1.5pt]{};
%     % Open dots
%     \draw (1.8, 0) circle (1.5pt);
%     \draw (3, 1) circle (1.5pt);
%     \draw (6, 0) circle (1.5pt);
%     \draw (7, 2) circle (1.5pt);
%     % indices
%     \node at (-0.2, -0.3) {$0$};
%     \node at (1.8, -0.4) {$t_1$};
%     \node at (3, -0.4) {$t_2$};
%     \node at (6, -0.4) {$t_3$};
%     \node at (7, -0.4) {$t_4$};
%     % Lines
%     \draw[domain=0:1.2, smooth, variable=\x, line width=1pt] plot ({\x}, {1.2-\x});
%     \draw[domain=1.2:1.76, smooth, variable=\x, line width=1pt] plot ({\x}, {0});
%     \draw[domain=1.8:2.98, smooth, variable=\x, line width=1pt] plot ({\x}, {4-\x});
%     \draw[domain=3:5, smooth, variable=\x, line width=1pt] plot ({\x}, {5-\x});
%     \draw[domain=5:5.97, smooth, variable=\x, line width=1pt] plot ({\x}, {0});
%     \draw[domain=6:6.98, smooth, variable=\x, line width=1pt] plot ({\x}, {9-\x});
%     \end{tikzpicture}
%     \caption{Sample path of the offered waiting time $V(t)$.}
%     \label{fig: sample path of the offered waiting time V(t)}
% \end{figure}

For the heavy-traffic analysis below, we follow the standing queueing assumptions of \cite{lee2011convergence}. In particular, initial customers may be taken to have infinite patience. Note that this is not a restrictive assumption but is a technical convention can be relaxed (see Lemma 2.1 of \cite{mandelbaum2012queues} and Lemma 4.1 of \cite{weerasinghe2014diffusion}). 

\subsection{A Sequence of Queueing Systems}
\label{Subsec2 of Sec1 in CH4: Sequence of Queueing Systems}

We now consider a sequence of single-server systems indexed by $n\geq 1$. The $n$th system operates on a faster time scale, where the arrival intensity is of order $n$, while individual service requirements are of order $1/n$. 
Let $V_n$ and $A_n$ denote the offered waiting-time process and the arrival process, respectively. 
Each component of the ordered triple $(t_j^n, v_j^n, d_j^n)$ denotes the arrival time, service time, and patience time of the $j$th customer in the $n$th system, respectively. We write the service requirement of the $j$th customer as $v_j^n := v_j/n$, where $(v_j)_{j\geq 1}$ is the i.i.d. sequence introduced above. The patience time of the $j$th customer $d_j^n$ has distribution function $F_n$. 
We introduce the discrete-time filtration $(\hat{\mathcal{F}}_i^n)_{i\geq0}$ by $\hat{\mathcal{F}}_0^n=\sigma(t_1^n)$ and
\begin{equation}
    \hat{\mathcal{F}}_i^n:=\sigma((t_1^n, v_1^n, d_1^n), \cdots, (t_i^n, v_i^n, d_i^n), t_{i+1}^n), 
    \label{discrete time filtration}
\end{equation}
for $i\geq1$. The associated continuous time filtration $(\mathcal{F}_t^n)_{t\geq0}$ is given by
\begin{equation}
    \mathcal{F}_t^n :=\hat{\mathcal{F}}_{[nt]}^{n}:=\sigma((t_1^n, v_1^n, d_1^n), \cdots, (t_{[nt]}^n, v_{[nt]}^n, d_{[nt]}^n), t_{{[nt]}+1}^n). 
    \label{continuous time filtration}
\end{equation}

The offered waiting-time process in the $n$th system is therefore
\begin{equation}
    V_n(t)
    =
    \frac{x_n}{\sqrt n}
    +
    \frac{1}{n}
    \sum_{j=1}^{A_n(t)}
    v_j\mathbbm 1_{\{V_n(t_j^n-)<d_j^n\}}
    -
    \int_0^t
    \mathbbm 1_{\{V_n(s)>0\}}\,ds,
    \quad  t\geq0,
    \label{offered waiting time for the nth system}
\end{equation}
where $x_n/\sqrt n$ is the initial offered waiting time and $x_n \to x\geq0$ as $n\to\infty$.

The system manager controls the arrival intensity through a state-dependent feedback policy $u_n:[0,\infty)\to\mathbb R$, where the argument of $u_n$ is the diffusion-scaled offered waiting time $\hat V_n(t):=\sqrt n\,V_n(t)$.
Given a feedback policy $u_n$, we define the corresponding arrival-intensity function by
\begin{equation}
    \lambda_n(v) := 1-\frac{1}{\sqrt n}u_n(\sqrt n\,v),
    \quad  v\geq0.
    \label{eq: controlled arrival intensity}
\end{equation}
Accordingly, when the offered waiting time is $V_n(t)$, the instantaneous arrival intensity becomes $n \lambda_n(V_n(t)) = n - \sqrt{n} u_n (\sqrt{n} V_n(t))$. 
Observe that at each state $x$, the system manager selects the intensity perturbation $u_n(x)$, which determines the arrival intensity through \eqref{eq: controlled arrival intensity} with $v = x/\sqrt{n}$. 
Thus, $u_n$ is the decision variable, whereas $\lambda_n$ is the arrival-intensity function induced by that decision. In particular, $u_n = 0$ corresponds to the nominal arrival intensity $n$, $u_n > 0$ reduces the arrival intensity below $n$, and $u_n <0$ increases it above $n$. Therefore, the control acts on the predictable intensity of the arrival process, rather than directly accepting or rejecting individual arrivals. For this reason, we refer to $u_n$ as an arrival-intensity control. This differs our work with conventional admission control. 

We restrict attention to admissible feedback controls for which $\lambda_n$ is positive and satisfies the regularity and boundedness conditions required for the heavy-traffic approximation in \cite{lee2011convergence}. In particular, we have the following assumption. 
% \begin{assumption}\label{assumption: asymptotic control un}
%     The sequence of controls $u_n$ is assumed to be Borel measurable and locally Lipschitz continuous, with the induced arrival intensity \eqref{eq: controlled arrival intensity} satisfying $0 < \epsilon_0 < \lambda_n < C_0 < \infty$ for some constants $\epsilon_0, C_0 > 0$ independent of $n$. 
%     Additionally, there exists a nonnegative locally Lipschitz continuous function $u:[0,\infty)\to [0, \infty)$ such that, for every $K>0$,
% \begin{equation}
%     \sup_{x\in[0,K]} |u_n(x)-u(x)|\to 0,
%     \quad  n\to\infty. 
%     \label{eq: control convergence}
% \end{equation}
% For the patience-time distributions, 
% % we specialize the general scaling framework of \cite{lee2011convergence} to the linear limiting abandonment regime. More precisely, 
% we assume that, for every $K>0$,
% \begin{equation}
%     \sup_{x\in[0,K]}
%     \left|
%         \sqrt n
%         F_n\left(\frac{x}{\sqrt n}\right)
%         -\theta x
%     \right|
%     \to 0,
%     \quad  n\to\infty,
%     \label{eq: patience scaling}
% \end{equation}
% for some $\theta\geq0$. 
% \end{assumption}

\begin{assumption}
\label{assumption: asymptotic control un}
The sequence of controls $\{u_n\}_{n\geq1}$ is Borel measurable and locally Lipschitz continuous. The induced arrival intensity \eqref{eq: controlled arrival intensity} satisfies $0<\epsilon_0\leq \lambda_n(x)\leq C_0<\infty$, for some constants $\epsilon_0,C_0>0$ independent of $n$. 

There exists a nonnegative locally Lipschitz continuous function
$u:[0,\infty)\to[0,\infty)$ such that, for every $K>0$,
\begin{equation}
    \sup_{x\in[0,K]}
    \lvert u_n(x)-u(x) \rvert
    \to 0,
    \quad  n\to\infty.
    \label{eq: control convergence}
\end{equation}

In addition, for every $K>0$, we have 
\begin{equation}
    \frac{1}{\sqrt n}
    \sup_{x\in[0,\sqrt n K]}
    \lvert u_n(x) \rvert
    \to 0,
    \quad  n\to\infty,
    \label{eq: control sublinear growth}
\end{equation}
and there exist $\delta_0>0$ and $M>0$ such that
\begin{equation}
    \sup_{n\geq1}
    \sup_{x\in[0,\sqrt n\delta_0]}
    [-u_n(x)]^+
    \leq M.
    \label{eq: control negative part bound}
\end{equation}

For the patience-time distributions, we assume that, for every $K>0$,
\begin{equation}
    \sup_{x\in[0,K]}
    \left\lvert
        \sqrt n
        F_n\left(\frac{x}{\sqrt n}\right)
        -\theta x
    \right\rvert
    \to 0,
    \quad  n\to\infty,
    \label{eq: patience scaling}
\end{equation}
for some $\theta>0$.
\end{assumption}

\begin{remark}
The additional growth conditions on $\{u_n\}$ are imposed to ensure that the induced arrival intensities $\{\lambda_n\}$ satisfy Assumption~3.2(ii)-(iii) of \cite{lee2011convergence}. In particular, for each $K > 0$, we have $\lim_{n\to\infty}\sup_{v\in[0,K]} \lvert\lambda_n(v)-1\rvert \to0$, while $\sqrt n(\lambda_n(v)-1)^+ = [-u_n(\sqrt n v)]^+$. Hence, these conditions allow us to apply the corresponding weak convergence results to the controlled offered waiting-time processes later in Section~\ref{Subsec3 of Sec1 in CH4: Weak Convergence of Queueing Systems}.
\end{remark}

To further analyze the $n$th queueing system \eqref{offered waiting time for the nth system}, it is helpful to introduce some fluid-scaled and diffusion-scaled quantities: for $t\geq 0$, 
\begin{equation}
% \begin{aligned}
    \bar{A}_n(t):=\frac{A_n(t)}{n}, \quad
    \hat{A}_n(t):=\frac{1}{\sqrt{n}}\left(A_n(t)-n\int_0^t\lambda_n(V_n(s))ds\right).  
    \label{scaled A quantities}
% \end{aligned}
\end{equation}
% for all $t\geq0$. 
We also define two diffusion-scaled martingales 
% of \eqref{M^v} and \eqref{M^d} 
% (see (3.7) in \cite{lee2011convergence}) 
with respect to the filtration $(\mathcal{F}_t^n)$: 
% introduced in \eqref{continuous time filtration}:
\begin{equation}
\begin{aligned}
    \hat{M}_n^v(t) &:= \frac{1}{\sqrt{n}}\sum_{j=1}^{[nt]}(v_j-1)\mathbbm{1}_{[V_n(t_j^n-)<d_j^n]}, \text{ and }\\
    \hat{M}_n^d(t) &:= \frac{1}{\sqrt{n}}\sum_{j=1}^{[nt]}\left(\mathbbm{1}_{[V_n(t_j^n-)\geq d_j^n]}-\mathbb E(\mathbbm{1}_{[V_n(t_j^n-)\geq d_j^n]} \mid \hat{\mathcal{F}}_{j-1}^n)\right),
    \label{two diffusion-scaled martingales}
\end{aligned}
\end{equation}
which enable us to construct a martingale representation of the $n$th system (see \cite{reed2008approximating}). 
% which enable us to construct an analogous martingale representation of the $n$th system similar to \eqref{M^v} and \eqref{M^d} in \eqref{centered offered waiting time process}. 
By \eqref{offered waiting time for the nth system}-\eqref{two diffusion-scaled martingales} and the idea of the Riemann-Stieltjes integral, a simple algebraic manipulation yields the diffusion-scaled state process: 
\begin{equation}
\begin{aligned}
    \hat{V}_n(t)  = x_n &+ \hat{A}_n(t) + \hat{M}_n^v(\bar{A}_n(t)) - \hat{M}_n^d(\bar{A}_n(t)) \\ 
    &-\int_0^t u_n(\hat{V}_n(s))ds - \frac{1}{\sqrt{n}} \int_0^t F_n\left(\frac{\hat{V}_n(s-)}{\sqrt{n}}\right) dA_n(s) + \hat{L}_n(t),
    \label{queueing systems}
\end{aligned}
\end{equation}
where $u_n(\hat{V}_n(\cdot))$ denotes the feedback action selected by the system manager based on the diffusion-scaled offered waiting time such that 
\begin{equation}
    u_n(x)=\sqrt{n}\left(1-\lambda_n\left(\frac{x}{\sqrt{n}}\right)\right), 
    \label{eq: u_n control}
\end{equation}
for all $x\geq0$, the state process $\hat{V}_n(\cdot)$ is the offered waiting time process scaled by $\sqrt{n}$, and the scaled idle time process is given by 
\begin{equation*}
    \hat{L}_n(t) = \sqrt{n}\int_0^t \mathbbm{1}_{[\hat{V}_n(s) = 0]}ds, 
\end{equation*}
for all $t\geq0$. 
Note that \eqref{queueing systems} makes explicit how the manager's arrival-intensity decision enters the diffusion-scaled waiting-time dynamics, where the $O(\sqrt n)$ perturbation of the physical arrival intensity produces the order-one drift term.

% \begin{remark}
% % [Arrival rates]
%     We assume the arrival rate has state-dependent intensity, which further boils down to the definition of waiting time \eqref{offered waiting time process}.
%     To scrutinize it, we define a discrete filtration: 
%     \[
%         \mathring{\mathcal{F}}_m^n = \sigma((t_1^n, v_1^n, d_1^n), \cdots, (t_m^n, v_m^n, d_m^n)). 
%     \]
%     It is noteworthy that $A_n(t)$ is a stopping time with respect to the filtration $(\mathring{\mathcal{F}}_m^n)_{m\geq 0}$, where $A_n(\cdot)$ is the arrival process with arrival times $(t_j^n)$ and $\mathring{\mathcal{F}}_m^n\subset \hat{\mathcal{F}_m^n}$ for all $n\geq 0$ as defined in \eqref{discrete time filtration}. 
%     Then, we define $\tilde{\mathcal{G}}_t^n:= \mathring{\mathcal{F}}_{A_n(t)}^n$ for all $t\geq 0$, which can be thought of as the information available to the system manager over time. 
%     Hence, for all $t\in[t_i^n, t_{i+1}^n)$, the quantity $\int_0^t \lambda_n(V_n(s))ds$ is known by time $t_i^n$. 
%     Observe that once $V_n(t_i^n)$ is known, $V_n(t)$ is well defined on the next interval $[t_i^n, t_{i+1}^n)$ since we assumed $\{A_n(t) - n\int_0^t \lambda_n(V_n(s))ds\}_{t\geq 0}$ is a $\tilde{\mathcal{G}}_t^n$-martingale. 
%     For a similar definition and more details of the waiting time process, one may refer to \cite{ward2005diffusion}, 
%     \cite{mandelbaum2008queues}, \cite{reed2008approximating}, \cite{dai2010customer}, \cite{lee2011convergence}. 
    
% \end{remark}

%%%%%%%%%%%%%%%%%%%%%%%%%%%%%%%%%%%%%%%%%%%%%%%%%%%%%%
\subsection{QCP Formulation}

We next formulate the control objective for the $n$th queueing system. 
The performance criterion consists of two components: a operational running cost and an idle server cost. 
The operating control cost is characterized by a non-trivial cost function related to the use of arrival-intensity control, and the idle server cost is generated by the penalty of idleness. 
% Consider an ordered quadruple $\{(x_n, u_n, \hat{V}_n, \hat{L}_n)\}_{n\geq1}$. 
% We introduce a cost structure that admits two types of costs: a control cost related to the use of control, which accumulates a control cost of $C(u_n(\hat{V}_n(t)))dt$, and a penalty for idleness cost, which accumulates a cost proportional to $d\hat{L}_n(t)$. 
% Here, the control cost is characterized by a cost function $C(\cdot)$, which satisfies Assumption \ref{ass: cost function} below, and the use of control $u_n$ in the $n$th system. 

% \begin{assumption}\label{ass: cost function}
%     The cost function $C(\cdot)$ is assumed to be a twice continuously differentiable, non-increasing, and convex function on $\mathbb{R}$. 
%     For $u\in [0, \infty)$, it is assumed that 
%     \[
%     0 \leq C(u) \leq M, 
%     \]
%     for some constant $M > 0$. 
%     For $u\in(-\infty, 0)$, we assume that there exists a constant $K_1> 0$ and an integer $l\geq 1$ such that 
%     \[
%     0\leq C(u) \leq K_1 (1 + |u|^l). 
%     \]
% \end{assumption}

\begin{assumption}
\label{assumption: prelimit control costs}
For each $n\geq1$, let $C_n:\mathbb R\to[0,\infty)$ be a continuous operational cost function associated with the control $u_n$. 
Assume that there exists a continuous
function $\bar C:\mathbb R\to[0,\infty)$ such that, for every $K>0$,
\begin{equation}
    \lim_{n\to\infty}
    \sup_{u\in[-K,K]}
    \lvert C_n(u)-\bar C(u) \rvert=0.
    \label{eq: local uniform convergence costs}
\end{equation}
Furthermore, there exist constants $K_1>0$ and an integer $\ell\geq1$, independent of $n$, such that
\begin{equation}
    0\leq C_n(u)\leq K_1(1+\lvert u \rvert^\ell),
    \quad  u\in\mathbb R,\quad n\geq1.
    \label{eq: uniform growth costs}
\end{equation}
The restriction $C:=\bar C\mid_{[0,\infty)}$ is assumed to be twice continuously differentiable, strictly decreasing, and to satisfy $C''(u) > 0$ and $\lim_{u\to\infty}C(u)=0$. 
\end{assumption}

\begin{remark}
The assumptions on the prelimit control costs admit, for example, the
following choices:
\[
C_n(u)=
\begin{cases}
c_0(1-\beta u), & u<0,\\
c_0e^{-\beta u}, & u\geq0,
\end{cases}
\]
and
\[
C_n(u)=
\begin{cases}
c_0(1-r\beta u), & u<0,\\
\dfrac{c_0}{(1+\beta u)^r}, & u\geq0,
\end{cases}
\]
where $c_0,\beta,r>0$. In either case, the sequence may be taken to be
independent of $n$. Its locally uniform limit $\bar C$ has restriction
to $[0,\infty)$ given respectively by
\[
C(u)=c_0e^{-\beta u}
\quad\text{or}\quad
C(u)=\frac{c_0}{(1+\beta u)^r}.
\]
Thus, although the prelimit controls $u_n$ may take negative values,
the limiting running cost on the diffusion-control domain $u\geq0$
retains the desired decreasing and strictly convex structure.
\end{remark}

% One may observe that if the system tends to adopt negative control of the $n$th system, then the arrival rate becomes large. Strong control of $u_n$, which in turn reduces the arrival rate, typically lowers congestion and thus decreases holding costs, though at the expense of increased server idleness. Conversely, weak control (i.e., allowing higher arrival rates up to $C_0$) leads to greater congestion, raising holding costs while reducing idleness costs. In both cases, control directly influences the trade-off between these costs. This motivates the use of a strictly decreasing cost function, as stated in Assumption \ref{ass: cost function}, which penalizes weak control more heavily. At the process level, this mechanism manifests as control $u_n$ driving the waiting time process toward zero, thereby activating the idle time process.

Since the control $u_n$ represents a diffusion-scaled perturbation of the nominal arrival intensity, in particular, positive values of $u_n$ reduce the arrival intensity, whereas negative values increase it. Thus, stronger positive control tends to reduce congestion and drive the offered waiting time toward the origin, but this may result in more frequent server idleness. Conversely, smaller, and possibly negative, values of $u_n$ correspond to higher arrival intensities and tend to reduce server idleness.
The running cost $C_n(u_n)$ represents the operating cost associated with the arrival-intensity control in the $n$th system. As specified in Assumption~\ref{assumption: prelimit control costs}, the sequence $(C_n)$ converges to a limiting cost whose restriction $C$ to $[0,\infty)$ is strictly decreasing. Consequently, in the limiting control problem, smaller values of the control incur a larger running cost, whereas stronger positive control reduces the running cost but tends to increase the idleness penalty. The control problem therefore captures a trade-off between the running control cost and the cost generated by server idleness.

For fixed $\alpha > 0$ and $p > 0$, we introduce the infinite-horizon discounted cost functional associated with \eqref{queueing systems} as follows: 
\begin{equation}
    J_n(x_n, u_n) = \mathbb E\left[\int_0^{\infty} e^{-\alpha t} \left(C_n(u_n(\hat{V}_n(t)))dt + pd\hat{L}_n(t)\right)\right], 
    \label{cost functional of QCP}
\end{equation}
where interest rate $\alpha > 0$ and $p>0$ are fixed constants. 
% From the perspective of the state dynamics in \eqref{queueing systems}, the control $u_n$ regulates the trade-off between the running control cost and server idleness. Positive values of $u_n$ reduce the arrival intensity and tend to drive the offered waiting time toward the origin, thereby increasing the frequency of server idleness. Conversely, smaller, and possibly negative, values of $u_n$ correspond to higher arrival
% intensities and tend to keep the system away from the origin, thereby reducing idleness. 
% Thus, stronger arrival-intensity regulation may reduce the running cost associated with the control while increasing the idleness penalty, whereas weaker regulation has the opposite effect. 
The queueing control problem therefore seeks a state-dependent control $u_n^*$ that optimally balances these competing costs by minimizing $J_n(x_n,u_n)$.

\begin{remark}
\label{rem: negative controls}
When $\lambda_n$ exceeds one, the diffusion-scaled arrival-intensity perturbation \eqref{eq: u_n control} may take negative values for finite $n$. 
A negative value of $u_n$ corresponds to an arrival intensity above the nominal level, whereas a positive value corresponds to an arrival intensity below the nominal
level. This motivates defining the prelimit running cost $C_n$ on $\mathbb R$.
In contrast, the heavy-traffic condition in Assumption 3.2 of \cite{lee2011convergence} requires the locally uniform limit $u$ of $(u_n)$ to be nonnegative. Accordingly, the limiting diffusion control problem is formulated over $u\geq0$, and only the restriction of the limiting running cost to $[0,\infty)$ enters the diffusion control
problem and its associated HJB equation.
\end{remark}

% \begin{remark}
% Negative controls arise naturally under Assumption 3.2 of \cite{lee2011convergence} when $C_0 > 1$. In practice, however, system managers may prefer to enforce non-negative controls by setting $C_0 = 1$, ensuring $u_n \geq 0$. Under this restriction, the cost function can be simplified to focus solely on the positive control regime. Assumption \ref{ass: cost function} provides a general framework for modeling control costs. When negative controls are allowed, it is common to assume a polynomial or exponential decay in the cost function. We therefore present separate cost structures for positive and negative controls. If the system enforces $u_n \geq 0$, the cost function can be assumed to be bounded above (cf. \cite{xie2022topics} for a detailed treatment of non-negative controls).

% % Moreover, other assumptions in Assumption \ref{ass: cost function} allow us to invoke two results in \cite{lee2011convergence} to deduce Lemma \ref{second moment bound of L} in Section \ref{Sec3 in CH4: Asymptotic Optimality}. 
% % One may refer to equations (6.6), (6.7), and (6.9) in \cite{lee2011convergence}. 
% % These assumptions are required only for Lemma \ref{second moment bound of L} as well as the weak convergence of the cost functional in Section \ref{Sec3 in CH4: Asymptotic Optimality}, and are not used elsewhere in our results. 
% % We will emphasize these assumptions when presenting those results. 
% \end{remark}

%----------------------------------------------------
\subsection{Weak Convergence of Queueing Systems}
\label{Subsec3 of Sec1 in CH4: Weak Convergence of Queueing Systems}

We conclude this section by recording the diffusion approximation needed for the subsequent asymptotic control analysis. The underlying weak-convergence argument is based on the offered waiting-time approximation established in \cite{lee2011convergence}. 
We state the result here in the particular linear-abandonment regime since it identifies the limiting state equation used in the diffusion control problem. 
% We conclude this section by presenting the weak convergence of the queueing systems introduced above. 
% We acknowledge the prior work \cite{lee2011convergence}, which established the diffusion approximation of the appropriately scaled waiting time process in \eqref{queueing systems}. 
% Such convergence enables the translation of the optimal policy from the DCP to the QCP, forming the basis for deriving an asymptotically optimal strategy in Section \ref{Sec3 in CH4: Asymptotic Optimality}. 
The proof closely follows that of Theorem 4.10 in \cite{lee2011convergence}, under the assumption $x_n \equiv x = 0$.

\begin{proposition}[Diffusion approximation] 
Suppose $\lim_{n\to\infty}x_n = x$, the feedback controls satisfy \eqref{eq: control convergence}-\eqref{eq: control negative part bound}, and the patience-time distributions satisfy \eqref{eq: patience scaling}. 
% , together with the remaining heavy-traffic conditions of \cite{lee2011convergence}. 
Then the process $(\hat{V}_n, \hat{L}_n)$ converges weakly to $(X, L)$ as $n\to\infty$ in $ D^{2}[0, \infty)$, where $(X, L)$ is the unique strong solution to the reflected stochastic differential equation
\begin{equation}
    % X(t) = x + \sigma W(t) - \int_0^t u(X(s))ds-\int_0^t H(X(s))ds + L(t), 
    X(t) = x + \sigma W(t) - \int_0^t u(X(s))ds- \theta \int_0^t X(s)ds + L(t), 
    \label{limiting process}
\end{equation}
for all $t\geq0$, where $W$ is a standard Brownian motion, and $\sigma>0$ is a constant and $\sigma^2=1+\sigma_s^2$, 
% The function $u(\cdot)$, defined on $[0, \infty)$, is locally Lipschitz continuous, as specified in Assumption 3.2 of \cite{lee2011convergence}, and it is real valued. 
% The process $X$ is non-negative with continuous sample paths, and 
the regulator $L$ denotes the regulator of $X$ at the origin. 
% The process $L(\cdot)$ is unique, continuous, non-decreasing process such that $L(0)=0$ and $\int_0^t X(s)dL(x)=0$
% \begin{equation*}
%     \int_0^t X(s)dL(x)=0,
% \end{equation*}
% for all $t\geq0$, and that $X(\cdot)\geq0$. 
\label{diffusion limit}
\end{proposition}

The diffusion process \eqref{limiting process} motivates the diffusion control problem studied in the next section. We will subsequently show that an optimal feedback policy for the limiting problem can be implemented in the prelimit queueing systems to construct an asymptotically optimal sequence of arrival-intensity controls.

\section{Diffusion Control Problem (DCP)}
\label{Sec3 in CH4: Diffusion Control Problem (DCP)}

\subsection{Problem Formulation}
\label{Subsec1 of Sec2 in CH4: Problem Formulation}

The heavy-traffic result in Proposition~\ref{diffusion limit} suggests that the diffusion-scaled offered waiting time can be approximated by a reflected diffusion whose drift is affected by the arrival-intensity perturbation. Motivated by this limiting structure, we formulate a diffusion control problem (DCP) that serves as a
tractable approximation to the queueing control problem (QCP) introduced in the previous section.

The role of the DCP is twofold. First, it provides a continuous-state control problem from which an optimal feedback policy can be characterized through the associated Hamilton-Jacobi-Bellman equation. Second, the resulting optimal diffusion control will later be used to construct a sequence of controls for the original queueing systems and to establish their asymptotic optimality later in Section~\ref{Sec3 in CH4: Asymptotic Optimality}. 

Accordingly, for an initial state $x\geq0$, we consider the controlled reflected diffusion
\begin{equation}
    X(t)
    =
    x+\sigma W(t)
    -\int_0^t u(s)\,ds
    -\theta\int_0^t X(s)\,ds
    +L(t),
    \quad  t\geq0,
    \label{control system}
\end{equation}
where $\sigma>0$ and $\theta>0$ are fixed constants. The term $-\theta X(t)$ represents the limiting effect of customer abandonment, while $u$ represents the diffusion-scale arrival-intensity control.  The nondecreasing process $L$ regulates $X$ at the origin so that $X(t)\geq0$, with
\[
    L(0)=0,
    \quad 
    \int_0^\infty X(t)\,dL(t)=0.
\]

To characterize the performance criterion, we assume that the DCP retains the two competing cost components of the queueing control problem. 
The non-trivial running cost $C(u(t))$ represents the operating cost associated with the arrival-intensity control, whereas the term $p\,dL(t)$ penalizes server idleness. Since $C$ is strictly decreasing, a larger control reduces the running cost but also introduces a stronger negative drift in the state process, causing the process to reach the origin more frequently and thereby increasing the idleness cost. Conversely, a smaller control produces a larger running cost but reduces the tendency of the system to become idle. The optimal policy therefore balances these two competing effects.

We define the infinite-horizon discounted cost functional associated with the controlled process \eqref{control system} by
\begin{equation}
    J(x, u) = \mathbb E \left[ \int_{0}^{\infty} e^{-\alpha t} \left(C(u(t))dt + pdL(t)\right) \right], 
    \label{costfunctional}
\end{equation}
% \begin{equation}
% % discretetize for RL
% \begin{aligned}
%     J(x, u, X) &= 
%     E\left[
%         \sum_{k=0}^{N-1} \int_{t_k}^{t_{k+1}} e^{-\alpha t} C(u(t))dt + \int_{t_k}^{t_{k+1}} e^{-\alpha t} dL(t)
%     \right] \\
%     &= E\left[
%     \sum_{t=0}^{N-1} \dfrac{e^{-\alpha t_{k}} C(u(t_k)) + e^{-\alpha t_{k+1}} C(u(t_{k+1})}{2} dt
%     + p \sum_{t=0}^{N-1} \frac{e^{-\alpha t_{k}} + e^{-\alpha t_{k+1}}}{2} (L(t_{k+1} - L(t_{k}))
%     \right]
% \end{equation}
where $\alpha>0$ is the discount factor and $p>0$ is the unit cost associated with server idleness. 
The objective of a system manager is to maintain the state process $X(\cdot)$ by varying the control $u$ so that it stays in a low-cost region.

An admissible control is a nonnegative measurable process $u=\{u(t):t\geq0\}$ such that
\[
    \int_0^T u(t)\,dt<\infty
    \quad  \text{a.s. for every }T>0,
\]
and for which the reflected state equation \eqref{control system} admits a solution. For $x\geq0$, we denote an ordered triple $(x, u, X)$ on the probability space $(\Omega, \mathcal{F}, P)$ an admissible control system if an $\mathcal{F}_t$-adapted control process $u(\cdot)$ satisfies the above assumptions and a corresponding state process $X(\cdot)$ obtained from \eqref{control system} exists on $[0, \infty)$. 

Let $\mathcal{A}(x)$ denote the collection of all such admissible control systems $(x, u, X)$. This set is non-empty since \eqref{control system} without control has a unique weak solution (see Theorem 5.2.1 in \cite{oksendal2013stochastic}). 
The value function of the DCP is defined by
\begin{equation}
    V(x) = \inf_{(x, u, X)\in \mathcal{A}(x)} J(x, u). 
    \label{value function}
\end{equation}
% where the infinite-horizon discounted cost functional $J$ is given by \eqref{costfunctional}. 
Notice that the value function is well defined since the cost functional associated with zero control is finite.  
(see \cite{xie2022topics}). 

Hence, our objective is to characterize an optimal admissible control $u^*$ and the corresponding controlled state process $X^*$ such that 
\[
J(x,u^*)=V(x).
\]
We accomplish this through the associated Hamilton-Jacobi-Bellman (HJB) equation. Specifically, we first establish the existence of an appropriate smooth solution to the HJB equation and then use a verification argument to identify the optimal feedback control $u^*$ and show that the HJB solution coincides with the value function $V$.
The resulting diffusion-optimal policy will subsequently be used to construct a sequence of controls for the original queueing systems in the asymptotic optimality analysis.

%----------------------------------
\subsection{Analysis of the HJB equation}
\label{sec: analysis of the HJB equation}

We introduce the formal Hamilton-Jacobi-Bellman (HJB) equation associated with the DCP \eqref{control system} by 
\begin{equation}
% \begin{aligned}
    \inf_{u\geq 0} \left\{ \frac{\sigma^2}{2} Q''(x) - uQ'(x) + C(u)- \theta x Q'(x)  - \alpha Q(x) \right\} = 0, 
    \quad Q'(0) = -p. 
    \label{formal HJB}
% \end{aligned}
\end{equation}
We now present the main result of this section, with the remainder devoted to its discussion. The technical proof is deferred to Appendix \ref{Appendix B}.

\begin{theorem}
There is a bounded $Q\in C^2[0, \infty)$ satisfying:
\begin{enumerate}[(i)]
    \item $Q$ is positive, strictly decreasing, convex on $[0, \infty)$, 
    % and non-negative if $-p \geq C'(u_0)$, where $u_0C'(u_0) = C(u_0)$, 
    
    \item There is an $M_0>0$ such that $-M_0\leq Q'(x)<0$ for all $x\geq 0$, and
    
    \item $Q$ satisfies the formal HJB equation \eqref{formal HJB} on $[0, \infty)$.
\end{enumerate}
\label{A solution to the HJB equation}
\end{theorem}

To tackle the HJB, we introduce the Legendre-Fenchel transform $F$: 
\begin{equation}
    % F(y) = \sup_{u\geq 0} \{uy-C(u)\}, 
    F(y) = \sup_{u\geq 0} \{uy-C(u)\}, 
    \label{Legendre transform}
\end{equation}
for all $y\in\mathbb{R}$, where $C(\cdot)$ is the non-trivial cost function satisfying Assumption \ref{assumption: prelimit control costs}. 
This transformation enables us to reformulate the formal HJB equation and contributes directly to its solution. Accordingly, we present its explicit form under our cost function assumptions and derive its analytical expression. 

% Since $C$ is strictly decreasing and strictly convex, we have $C'(u) < 0$ and $C''(u) > 0$ for all $u\geq 0$. 
Note that we have $C'(u) < 0$ and $C''(u) > 0$ for all $u\geq 0$.
Moreover, the convexity of $C$ together with $\lim_{u\to\infty}C(u) = 0$ implies that $\lim_{u\to\infty}C'(u) = 0$. 
Consequently, $C'$ is strictly increasing from $C'(0)$ to $0$ on $[0, \infty)$. 
For fixed $y < 0$, we define
\[
g_y(u) := uy - C(u), \quad u\geq 0. 
\]
Then, we observe that 
\[
g_y'(u) =y - C'(u), \quad g_y''(u) = -C''(u) < 0, 
\]
which suggest $g_y$ is strictly concave. 
If $y \leq C'(0)$, the maximum is attained at the boundary $u = 0$. If $C'(0) < y<0$, then the unique maximizer is the solution of $C'(u) = y$, namely $u = (C')^{-1}(y)$. Therefore, we have
\begin{equation}
F(y)
    =
    \begin{cases}
        +\infty, & y>0,\\
        0, & y=0,\\
        y(C')^{-1}(y)
        -C\!\left((C')^{-1}(y)\right),
        & C'(0)<y<0,\\
        -C(0), & y\leq C'(0).
    \end{cases}
    \label{Legendre transform in our case C}
\end{equation}
Furthermore, for $y<0$, one can easily compute its derivative as
\begin{equation}
    F'(y)
    =
    \begin{cases}
        0, & y\leq C'(0),\\
        (C')^{-1}(y), & C'(0)<y<0.
    \end{cases}
    \label{F'}
\end{equation}
In particular, $F'(y)\geq0$ for all $y<0$.

To develop a better understanding of the Legendre transform, we present some concrete examples of the cost function.

\begin{example}[Exponential control cost]
\label{example: exponential cost}
Let $C(u)=c_0e^{-\beta u}$ for $u\geq0$, where $c_0,\beta>0$. Its Legendre transform is
\begin{equation}
    F(y)
    =
    \begin{cases}
        +\infty, & y>0,\\
        0, & y=0,\\
        -\dfrac{y}{\beta}
        \ln\left(-\dfrac{y}{\beta c_0}\right)
        +\dfrac{y}{\beta},
        & -\beta c_0<y<0,\\
        -c_0, & y\leq-\beta c_0.
    \end{cases}
    \label{eq: exponential conjugate}
\end{equation}
Moreover, for $y<0$,
\begin{equation}
    F'(y)
    =
    \begin{cases}
        0, & y\leq-\beta c_0,\\
        -\dfrac{1}{\beta}
        \ln\left(-\dfrac{y}{\beta c_0}\right),
        & -\beta c_0<y<0.
    \end{cases}
    \label{eq: exponential conjugate derivative}
\end{equation}
In Figure~\ref{fig: exponential cost LF transform}, we can visualize its structure for the case of $c_0 = 2$ and $\beta = 1$. 
\end{example}
\begin{figure}[htbp]
    \centering

    %====================================================
    % (a) Legendre transform F(y)
    %====================================================
    \begin{subfigure}[t]{0.47\textwidth}
        \centering
        \begin{tikzpicture}
            \begin{axis}[
                width=\textwidth,
                axis lines=middle,
                xmin=-5, xmax=0.3,
                ymin=-2.4, ymax=0.4,
                xtick={-2,0},
                ytick={-2,0},
                samples=250,
            ]

                % y <= -2
                \addplot[thick, domain=-5:-2] {-2};

                % -2 < y < 0
                \addplot[
                    thick,
                    domain=-1.999:-0.01
                ]
                {-x*ln(-x/2)+x};

                \addplot[only marks, mark=*] coordinates {(-2,-2)};
                \addplot[only marks, mark=*] coordinates {(0,0)};

                \draw[dashed]
                    (axis cs:-2,-2.4) --
                    (axis cs:-2,0.25);

                % \node[anchor=south west] at (axis cs:-4.8,-1.85)
                    % {$F(y)=-2$};

            \end{axis}
        \end{tikzpicture}
        \caption{Legendre transform $F(y)$.}
        \label{fig: exponential F}
    \end{subfigure}
    \hfill
    %====================================================
    % (b) Derivative F'(y)
    %====================================================
    \begin{subfigure}[t]{0.47\textwidth}
        \centering
        \begin{tikzpicture}
            \begin{axis}[
                width=\textwidth,
                axis lines=middle,
                xmin=-5, xmax=0.3,
                ymin=-0.2, ymax=4.5,
                xtick={-2,0},
                ytick={0},
                samples=250,
            ]

                % y <= -2
                \addplot[thick, domain=-5:-2] {0};

                % -2 < y < 0
                \addplot[
                    thick,
                    domain=-1.999:-0.03
                ]
                {-ln(-x/2)};

                \addplot[only marks, mark=*] coordinates {(-2,0)};

                \draw[dashed]
                    (axis cs:-2,-0.2) --
                    (axis cs:-2,4.3);

            \end{axis}
        \end{tikzpicture}
        \caption{Derivative $F'(y)$.}
        \label{fig: exponential Fprime}
    \end{subfigure}

    \caption{
    The Legendre transform and its derivative for the exponential running cost $C(u)=2e^{-u}$.
    }
    \label{fig: exponential cost LF transform}
\end{figure}

\begin{example}[Polynomially decaying control cost]
\label{example: polynomial cost}
Let $C(u) =\frac{c_0}{(1+\beta u)^r}$ for $u\geq0$, where $c_0,\beta,r>0$. Then its Legendre transform is 
\begin{equation}
    F(y)
    =
    \begin{cases}
        +\infty, & y>0,\\
        0, & y=0,\\
        \displaystyle
        \frac{y}{\beta}
        \left[
            \left(
                -\frac{c_0r\beta}{y}
            \right)^{\frac{1}{r+1}}
            -1
        \right]
        -
        c_0
        \left(
            -\frac{y}{c_0r\beta}
        \right)^{\frac{r}{r+1}},
        &
        -c_0r\beta<y<0,\\
        -c_0,
        & y\leq-c_0r\beta.
    \end{cases}
    \label{eq: polynomial conjugate}
\end{equation}
Moreover,
\begin{equation}
    F'(y)
    =
    \begin{cases}
        0,
        & y\leq-c_0r\beta,\\ 
        \displaystyle
        \frac{1}{\beta}
        \left[
            \left(
                -\frac{c_0r\beta}{y}
            \right)^{\frac{1}{r+1}}
            -1
        \right],
        &
        -c_0r\beta<y<0.
    \end{cases}
    \label{eq: polynomial conjugate derivative}
\end{equation}
\end{example}

With the help of the Legendre-Fenchel transform introduced in \eqref{Legendre transform} and together with HJB \eqref{formal HJB}, it suffices to find a solution $Q(\cdot)$ satisfying
\begin{equation}
% \begin{aligned}
     \frac{\sigma^2}{2} Q''(x) - F(Q'(x)) - \theta x Q'(x) - \alpha Q(x)  = 0, \quad 
    Q'(0) = -p. 
    % Q'(x)&<0, 
    \label{rewritten HJB}
% \end{aligned}
\end{equation}
We seek a bounded classical solution $Q\in C^2[0,\infty)$ satisfying $Q'(x)<0$ for all $x\geq0$. 
% Under Assumption \ref{assumption: prelimit control costs}, the limiting cost $C=\bar C|_{[0,\infty)}$ is strictly decreasing and strictly convex. Hence, for $y<0$, the maximizer in \eqref{Legendre transform} is uniquely determined. Therefore, $F'(Q'(x))$ provides the candidate feedback control associated with a solution of \eqref{rewritten HJB}.

Note that \eqref{rewritten HJB} is a second-order nonlinear ordinary differential equation on the half-line, whereas the reflecting boundary at the origin provides only the condition $Q'(0)=-p$. The appropriate solution must therefore be selected by its global behavior on $[0,\infty)$.  
More precisely, since \eqref{rewritten HJB} is second order while only the boundary condition $Q'(0)=-p$ is prescribed, we parameterize the family of solutions by
\[
    Q''(0)=r.
\]
The differential equation then determines the corresponding value $Q(0)$ through
\[
    Q(0)
    =
    \frac{1}{\alpha}
    \left(
        \frac{\sigma^2}{2}r-F(-p)
    \right).
\]
In Appendix \ref{Appendix B}, we analyze a family of solutions parameterized by the unspecified initial curvature and identify the unique parameter $r=r^*$ for which the resulting solution has the required boundedness, monotonicity, and convexity properties. This yields the classical solution used in the verification argument of the next section.

%----------------------------------
\subsection{Optimal Solution of the DCP}
\label{Subsec2 of Sec3 in CH4: Optimal Solution of the DCP}

We now identify the value function and an optimal control for the DCP. The argument consists of two steps. First, a verification lemma shows that the solution of the HJB equation provides a lower bound for the value function. We then construct a feedback control that attains this lower bound. Together with Theorem~\ref{A solution to the HJB equation}, this establishes the optimality of the candidate solution obtained in the previous section.

% %%%%%%%%%%%%%%%%%%%%%%%%

% Now with all the tools in hand, we are ready for the verification lemma which indicates that the value function has a lower bound. 

\begin{lemma}[Verification Lemma]
Let $Q$ be a bounded twice continuously differentiable solution to \eqref{formal HJB} on $[0, \infty)$ which satisfies $-M\leq Q'(x)<0$ for $x\in[0, \infty)$ and some constant $M>0$. 
Then $Q$ is a lower bound such that $V(x) \geq Q(x)$ for all $x\in [0, \infty)$. 
\label{verification lemma}
\end{lemma}

We omit this standard proof for brevity (for its full proof, see Appendix~\ref{appendix:verificaiton}). The argument primarily relies on employing It\^{o}'s formula to $f(t, x) = e^{-\alpha t} Q(x)$, where $Q$ satisfies the stated assumptions and $\alpha>0$ is a fixed constant. In conjunction with \eqref{formal HJB}, one may obtain the lower bound of the value function.

Since Theorem \ref{A solution to the HJB equation} guarantees the existence of $Q(\cdot)$, and Lemma \ref{verification lemma} establishes that $Q$ provides a lower bound for the value function \eqref{value function}, it remains to show that this bound is attainable. We do so by constructing a specific control $u^*$, which is then shown to be the optimal strategy. 

% For $y<0$, the supremum in the definition of $F$ is attained at $u=F'(y)$. Hence,
% \begin{equation}
%     F(y)=yF'(y)-C(F'(y)),
%     \quad  y<0.
%     \label{eq: Legendre optimizer identity}
% \end{equation}
% Since $Q'(x)<0$ for every $x\geq0$, 
We define the feedback control
\begin{equation}
    u^*(x):=F'(Q'(x)),
    \quad  x\geq0.
    \label{eq: optimal feedback}
\end{equation}
It is easy to observe that since $Q'(x)<0$ for all $x\geq0$ as in Theorem \ref{A solution to the HJB equation}, the feedback control $u^*(\cdot)$ is finite and well defined. 
One can see that there exists a weak solution to the stochastic differential equation \eqref{optimal SDE} for all $t\geq0$ (for more details, see \cite{xie2022topics}, where an interesting proof using a comparison result is provided). 

Substituting \eqref{eq: optimal feedback} into \eqref{formal HJB} gives
\begin{equation}
\begin{aligned}
    \frac{\sigma^2}{2}Q''(x)
    -u^*(x)Q'(x)
    -\theta xQ'(x)
    -\alpha Q(x)
    +C(u^*(x))
    =0.
    \label{rewritten HJB with F be substituted}
\end{aligned}
\end{equation}

We now state the optimality result. 
Here, we aim to show that the lower bound established by the verification lemma (Lemma \ref{verification lemma}) is attained by the control $u^*$. 
We begin by outlining the main idea of the proof; full technical details can be found in Appendix~\ref{appendix:verificaiton}. 
We apply Itô's formula to the function $e^{-\alpha t}Q(X^*(t))$. Substituting the expression from \eqref{rewritten HJB with F be substituted}, and using the boundedness of $Q'$ along with properties of the regulator process $L^*$, we take expectations on both sides and let $t\to\infty$ to obtain
\[
    Q(x) =  \mathbb E\left[\int_0^{\infty} e^{-\alpha s}C(F'(Q'(X^*_s)))ds + p\int_0^\infty e^{-\alpha s} dL_s^*\right]
    % & = E\left[\int_0^{\infty} e^{-\alpha s}\left(C(u^*(s))ds+pdL_s^*\right)\right],  
% \end{aligned}
\]
which coincides with the cost functional $J(x, u^*)$ defined in \eqref{costfunctional}. 
% Observe that the right-hand side is identical with the cost functional $J(x, u^*, X^*)$ as defined in \eqref{costfunctional}. 

It is straightforward to see that the lower bound can be achieved by assigning the control $u^*$. Consequently, $V(x) = Q(x)$ for all $x\in[0, \infty)$. 
Since the uniqueness of a smooth solution to \eqref{formal HJB}, it follows that $V(\cdot)$ coincides with $Q(\cdot)$ and satisfies the formal HJB. 
Additionally, any such solution $Q$ to \eqref{formal HJB} must be non-negative. 
We summarize this result in the following theorem. 

\begin{theorem}
\label{An optimal control}
Let $Q$ be the solution obtained in Theorem~\ref{A solution to the HJB equation}. 
Suppose $X^*$ is a weak solution of the reflected stochastic differential equation
\begin{equation}
\begin{aligned}
    X^*(t)
    ={}&x+\sigma W(t)
    -\int_0^t u^*(X^*(s))\,ds
    -\theta\int_0^t X^*(s)\,ds
    +L^*(t),
\end{aligned}
\label{optimal SDE}
\end{equation}
where $L^*$ is the regulating process at the origin. Then $(x,u^*,X^*)$ is an optimal admissible control system and $V(x)=Q(x)$. In particular, \eqref{eq: optimal feedback} is an optimal control. 
\end{theorem}

%------------------------------------
\section{Asymptotic Optimality}
\label{Sec3 in CH4: Asymptotic Optimality}

The goal of this section is to show that the cost functional of the QCP, introduced in \eqref{cost functional of QCP}, converges to the corresponding cost functional of the DCP discussed in Section \ref{Sec3 in CH4: Diffusion Control Problem (DCP)}. 
In particular, it converges to the value function \eqref{value function}. This convergence allows us to ``translate" the optimal strategy of the DCP into a nearly optimal strategy for the QCP. 

The asymptotic optimality argument consists of two parts. First, we establish that the limiting DCP value function $V(x)$ provides an asymptotic lower bound for the costs of admissible competing sequences of queueing controls. Second, we construct a sequence of queueing controls from the diffusion-optimal feedback $u^*$ and show that the corresponding costs converge to $V(x)$. Consequently, the DCP-derived sequence attains the asymptotic lower bound and is therefore asymptotically optimal.

First, under the cost structure introduced in Section \ref{Subsec2 of Sec1 in CH4: Sequence of Queueing Systems}, we recall that the discounted cost associated with the $n$th queueing system:
\begin{equation}
\begin{aligned}
    J_n(x_n,u_n)
    &:=
    \mathbb E\left[
        \int_0^\infty e^{-\alpha t}
        C_n(u_n(\hat V_n(t)))\,dt
        +
        p\alpha\int_0^\infty e^{-\alpha t}\hat L_n(t)\, dt
    \right].
    \label{eq: QCP cost asymptotic section}
\end{aligned}
\end{equation}
Here, $C_n$ is the prelimit operational cost introduced in Assumption~\ref{assumption: prelimit control costs}.
Similarly, we recall the cost functional of DCP
\begin{equation}
    J(x, u) = \mathbb E\left[\int_0^{\infty} e^{-\alpha t} \left(C(u(X(t))) + p\alpha L(t) \right)dt\right]. 
    \label{eq: rewritten cost functional of limiting process}
\end{equation}
Here, the regulator cost has been rewritten since we have $\int_0^{\infty} \int_t^{\infty} \alpha pe^{-\alpha s}ds dL(t) = \int_0^{\infty}\int_0^s \alpha pe^{-\alpha s}dL(t)ds$ by Fubini-Tonelli theorem, i.e., \eqref{costfunctional} and \eqref{eq: rewritten cost functional of limiting process} are equivalent. 

% Next, we recall the value function of the diffusion control problem introduced in \eqref{value function}. We call an ordered quadruple $(x, u, X, L)$ on the probability space $(\Omega, \mathcal{F}, P)$ an admissible control system if an $\mathcal{F}_t$-adapted control process $u(\cdot)$ satisfies the Assumptions 3.2 in \cite{lee2011convergence} and a corresponding state process $X(\cdot)$ obtained from \eqref{limiting process} exists on $[0, \infty)$. Let $\mathcal{A}(x)$ be the collection of all such admissible control systems $(x, u, X, L)$. This set is non-empty since \eqref{limiting process} with $u=0$ has a unique weak solution (see Theorem 5.2.1 in \cite{oksendal2013stochastic}). 
% Thus, the corresponding value function of the diffusion control problem could be defined as
% \begin{equation}
%     V(x) = \inf_{(x, u, X, L)\in \mathcal{A}(x)} J(x, u, X, L). 
%     \label{eq: value function with x}
% \end{equation}
% % where the infinite-horizon discounted cost functional $J$ for the DCP is given by \eqref{eq: rewritten cost functional of limiting process}. 

The main results of this section are as follows and their proofs are deferred to Appendix~\ref{appendix: asymptotic optimality}. 

\begin{theorem}[Asymptotic lower bound]
     Let $\{u_n\}_{n\geq 1}$ be an admissible competing sequence satisfying Assumption~\ref{assumption: asymptotic control un}, with limiting feedback control $u$. Let $(\hat V_n, \hat L_n)$ denote the corresponding diffusion-scaled queueing processes. If $\lim_{n\to\infty}x_n = x$, then
    \begin{equation}
        \liminf_{n\to\infty} J_n(x_n, u_n) \geq J(x, u).  
        \label{eq: asymptotic lower bound first}
    \end{equation}
    Consequently, $\liminf_{n\to\infty} J_n(x_n, u_n) \geq V(x)$ holds, where $V(x)$ is the value function of the DCP defined in \eqref{value function}. 
    \label{thm: asymptotic lower bound}
\end{theorem}

Here, Theorem~\ref{thm: asymptotic lower bound} provides an asymptotic lower bound for the cost functional of the QCP, $J_n(x_n, u_n)$. It remains to show that this lower bound is achievable by taking a special feedback control $u^*$ introduced in \eqref{eq: optimal feedback}, which solves the DCP.

We now construct a sequence of admissible queueing controls from the optimal feedback $u^*$ of the DCP. Let $\{\kappa_n\}_{n\geq1}$ be a sequence satisfying $\kappa_n\to\infty$, $\kappa_n / \sqrt n\to0$, and $\sup_{0\leq u\leq\kappa_n}  \lvert C_n(u)-C(u)\rvert\to0$ as $n\to\infty$.
Define
\begin{equation}
    u_n^*(x):=u^*(x)\wedge\kappa_n,
    \quad  x\geq0.
    \label{eq: prelimit optimal feedback}
\end{equation}
Since $u^*(x)\geq0$, we have $0\leq u_n^*(x)\leq\kappa_n$. 
Moreover, on every compact subset of $[0,\infty)$, $u_n^*$ agrees with $u^*$ for all sufficiently large $n$. 
Together with $\kappa_n/\sqrt n\to0$, it follows that the sequence $\{u_n^*\}$ satisfies Assumption~\ref{assumption: asymptotic control un} for all sufficiently large $n$. Thus, $\{u_n^*\}$ provides an admissible prelimit implementation of the DCP-optimal feedback. 

\begin{remark}
\label{rem:kappa-existence}
Such a sequence $\{\kappa_n\}_{n\geq1}$ always exists under Assumption~\ref{assumption: prelimit control costs}.
By the local uniform convergence of $C_n$ to $C$ on $[0,\infty)$, for each integer $m\geq1$ there exists $N_m$, with $N_m<N_{m+1}$, such that, for all $n\geq N_m$, we have
\[
    \sup_{0\leq u\leq m}\lvert C_n(u)-C(u)\rvert
    \leq \frac1m,
    \qquad
    \frac{m}{\sqrt n}\leq \min \left\{\frac1m, \frac12\right\}.
\]
Setting $\kappa_n=m$ for $N_m\leq n<N_{m+1}$ gives the desired properties. 
For the finitely many indices preceding $N_1$, $\kappa_n$ may be chosen sufficiently small so that $\kappa_n/\sqrt n\leq 1/2$. 
Moreover, since $0\leq u_n^*(x)\leq\kappa_n$, the corresponding controlled arrival
intensity satisfies
\[
    1-\frac{u_n^*(x)}{\sqrt n}
    \geq
    1-\frac{\kappa_n}{\sqrt n}
    \geq\frac12,
\]
and hence remains uniformly bounded away from zero for every $n\geq1$.
\end{remark}

\begin{theorem}[Asymptotic optimality]
\label{Asymptotic optimality}
Assume $x_n\to x$, and let $\{u_n^*\}$ be the sequence of queueing controls constructed from the DCP-optimal feedback $u^*$. 
Let $(\hat V_n^*,\hat L_n^*)$ denote the corresponding scaled state and
idle-time processes. Then
\begin{equation}
    \lim_{n\to\infty}
    J_n(x_n,u_n^*)
    =
    J(x,u^*)
    =
    V(x).
    \label{eq: asymptotic optimal cost convergence}
\end{equation}
Consequently, $\{u_n^*\}_{n\geq1}$ is asymptotically optimal among the admissible competing sequences considered in Theorem \ref{thm: asymptotic lower bound}.
\end{theorem}

The proofs of these theorems rely on several preliminary results. The key idea is to verify moment bounds for each term to ensure uniform integrability. We defer the necessary preliminaries and supporting discussions to Appendix~\ref{appendix: asymptotic optimality}.

%%%%%%%%%%%%%%%%%%%%%%%%%%%%%%%%%%%%%%%%%%%%%%%%%
\section{Numerical Solution for the HJB Equation}
\label{sec:numerical}

In this section, we numerically investigate the solution of the HJB equation and the resulting optimal state-feedback intensity policy for the DCP. 
Our main objective is to implement thecritical-parameter construction developed in Appendix~\ref{Appendix B} to compute the value function and the corresponding optimal feedback control. 
All computations are based on the cost function introduced in Example~\ref{example: exponential cost}. 
The baseline parameter values used throughout the numerical experiments are reported in Table~\ref{tab:baseline_parameters}. 

\begin{table}[htbp]
\centering
\caption{Baseline parameter values used in the numerical experiments.}
\label{tab:baseline_parameters}
\begin{tabular}{ccc}
\hline
Parameter & Baseline value & Description\\
\hline
$\sigma$ & $1.0$ & Diffusion coefficient\\
$\theta$ & $0.5$ & Reneging parameter\\
$\alpha$ & $0.5$ & Discount factor\\
$p$      & $1.1$ & Idleness penalty\\
\hline
\end{tabular}
\end{table}

Additional numerical experiments are deferred to Appendix~\ref{appendix:additional-numerical}, where we examine the sensitivity of the optimal feedback policy with respect to the model parameters, and compare the DCP-optimal policy with selected competing policies. 
% and investigate numerically the performance of the DCP-derived controls when implemented in the prelimit queueing systems. 

% \subsection{Numerical Solution of the HJB Equation}
% \label{sec:numerical_hjb}

% We numerically compute the solution of the HJB equation and the corresponding optimal state-feedback control. 
% We follow the critical-parameter construction developed in Appendix~\ref{Appendix B}. This provides a numerical procedure that is consistent with the analytical characterization of the value function.

Recall the reduced HJB equation in \eqref{rewritten HJB}. For the numerical implementation, define
\[
Y_r(x):=Q(x),
\quad 
W_r(x):=Y_r'(x),
\]
with initial conditions
\[
Y_r'(0)=-p,
\quad 
Y_r''(0)=r,
\]
as in Appendix~\ref{Appendix B}. The HJB equation can then be written as the first-order system
\begin{equation}
\begin{bmatrix}
Y_r'(x)\\
W_r'(x)
\end{bmatrix}
=
\begin{bmatrix}
W_r(x)\\[1mm]
\dfrac{2}{\sigma^2}
\left[
F\bigl(W_r(x)\bigr)
+\theta xW_r(x)
+\alpha Y_r(x)
\right]
\end{bmatrix},
\label{eq:numerical_system}
\end{equation}
with initial condition
\[
\begin{bmatrix}
Y_r(0)\\
W_r(0)
\end{bmatrix}
=
\begin{bmatrix}
K_r\\
-p
\end{bmatrix}
=
\begin{bmatrix}
\dfrac{1}{\alpha}
\left(
\dfrac{\sigma^2}{2}r-F(-p)
\right)\\[2mm]
-p
\end{bmatrix}.
\]
Thus, for each fixed $r\geq0$, the numerical problem is formulated as an initial value problem (IVP) for the two-dimensional system \eqref{eq:numerical_system}.

The characterization of the critical parameter $r^*$ in \eqref{r^*} leads naturally to a numerical search procedure. For $r<r^*$, the trajectory $W_r$ develops a negative local maximum, characterized numerically by $W_r'(x)=0$ while $W_r(x)<0$. 
For $r>r^*$, the trajectory reaches zero at a finite state value. 
The critical trajectory corresponding to $r=r^*$ separates these two regimes. We therefore first identify one parameter value in each regime and then refine the resulting bracket by bisection. 

In Algorithm~\ref{alg:ivp_hjb} we summarize the numerical procedure. 
Here, a finite computational horizon $X_{\max}$ is used only to follow the IVP over a sufficiently large portion of the state space and to classify the corresponding trajectory. It does not impose a terminal boundary condition. This is consistent with Proposition~\ref{exist r^*}, according to which the critical solution satisfies $W_{r^*}(x)<0$, $W_{r^*}'(x)>0$, and $\lim_{x\to\infty}W_{r^*}(x)=0$. 
Thus, the derivative of the value function approaches zero asymptotically rather than being forced to equal zero at a finite truncation point. 

\begin{algorithm}[htbp]
\caption{IVP-based computation of $r^*$ and the optimal feedback control}
\label{alg:ivp_hjb}
\begin{algorithmic}[1]
\Require Parameters $\sigma,\theta,\alpha,p$; computational horizon $X_{\max}$; tolerance $\varepsilon_r>0$
\Ensure Approximation of $r^*$, value function $Q$, and optimal control $u^*$

\State Define $F$ and $F'$ as in Example~\ref{example: exponential cost}.

\For{each trial value $r\geq0$}
\State Compute $K_r=\frac{1}{\alpha}\left(\frac{\sigma^2}{2}r-F(-p)\right)$.
\State Solve the IVP~\eqref{eq:numerical_system} with $Y_r(0)=K_r$ and $W_r(0)=-p$.
\State Track the solution until $X_{\max}$ or until either $W_r'(x)=0$ with $W_r(x)<0$, or $W_r(x)=0$.
\EndFor

\State Choose $r_L$ in the local-maximum regime and $r_U$ in the zero-reaching regime.

\While{$r_U-r_L>\varepsilon_r$}
\State Set $r_M=(r_L+r_U)/2$.
\State Solve the IVP corresponding to $r_M$.
\If{$W_{r_M}$ attains a negative local maximum}
\State $r_L\gets r_M$.
\ElsIf{$W_{r_M}$ reaches zero}
\State $r_U\gets r_M$.
\Else
\State Extend the computational horizon and continue the IVP.
\EndIf
\EndWhile

\State Set $r^*\approx(r_L+r_U)/2$.
\State Solve \eqref{eq:numerical_system} using $r=r^*$ and set $Q(x)=Y_{r^*}(x)$ and $Q'(x)=W_{r^*}(x)$. 
\State Compute $u^*(x)=F'\bigl(Q'(x)\bigr)$. 
\end{algorithmic}
\end{algorithm}

For the numerical experiments below, we use the cost function in Example~\ref{example: exponential cost}. Once $r^*$ is obtained through exhaustive search, the optimal arrival-intensity control is evaluated directly from \eqref{eq: optimal feedback}. 
% \begin{equation*}
% u^*(x)=F'\bigl(Q'(x)\bigr).
% \end{equation*}
The numerical parameters used for the IVP computation are reported in Table~\ref{tab:hjb_parameters}.

\begin{table}[htbp]
\centering
\caption{Parameters for the HJB computation.}
\label{tab:hjb_parameters}
\begin{tabular}{ccc}
\hline
Parameter & Value & Description \\
\hline
$X_{\max}$        & $15$        & Computational horizon for determining $r^*$ \\
$\varepsilon_r$   & $10^{-8}$   & Bisection tolerance for $r^*$ \\
$\texttt{rtol}$   & $10^{-10}$  & Relative tolerance for the IVP solver \\
$\texttt{atol}$   & $10^{-12}$  & Absolute tolerance for the IVP solver \\
$\Delta x_{\max}$ & $0.005$     & Maximum IVP integration step \\
\hline
\end{tabular}
\end{table}

Figure~\ref{fig:critical_r} illustrates the role of the critical parameter by plotting $W_r=Y_r'$ for values of $r$ around $r^*$. Trajectories corresponding to $r<r^*$ develop a negative local maximum, whereas trajectories corresponding to $r>r^*$ reach zero at a finite state value. The critical trajectory separates these two behaviors and provides the numerical approximation of $Q'$ used to construct the optimal feedback policy. 
Note that the current implementation uses a grid-based exhaustive search to obtain an initial bracket for $r^*$, followed by bisection to refine the estimate. While robust, the initial search can be computationally inefficient since the IVP must be solved repeatedly for different trial values of $r$. 

\begin{figure}[htbp]
    \centering
    \includegraphics[width=0.65\linewidth]{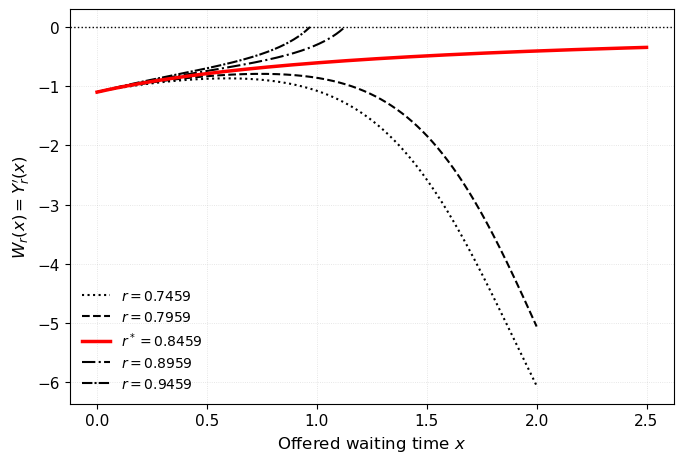}
    \caption{Numerical trajectories of $W_r$ for different values of $r$ around the critical parameter $r^*$.}
    \label{fig:critical_r}
\end{figure}

%%%%%%%%%%%%%%%%%%%%%%%%%%%%%%%%%%%%%%%%%%%%%%%%%

\section{Sample-Based Policy Learning for the DCP}\label{sec:RL}

The analysis in Section~\ref{Sec3 in CH4: Diffusion Control Problem (DCP)} characterizes the optimal feedback through the HJB equation. We now examine the performance of policies learned from simulated state transitions and realized costs, using the independently computed HJB feedback as an evaluation benchmark. Related connections between reinforcement learning and stochastic control are discussed in~\cite{quer2022connecting,sutton2018reinforcement}.

For the current one-dimensional model, the HJB approach requires full knowledge of the model
coefficients and a numerical determination of the boundary parameter $r^*$ in
\eqref{r^*}. Sample-based learning methods construct their policies through repeated trajectory generation and policy updates. Evaluating these policies in a common simulation environment provides a basis for assessing their performance and exploring extensions with simulator access or higher-dimensional HJB equations.

\subsection{Discrete-time control model}

Fix a time step $h>0$. We formulate the discretized control problem as the
discounted MDP confined to the control 
\begin{equation}
    \mathcal M_h=(\mathcal X,\mathcal U,P_h,c_h,\gamma_h),
    \qquad
    \mathcal X=\mathbb R_+, \quad \mathcal U=\mathbb [0,\infty),
    \label{eq:discrete-mdp}
\end{equation}
where $P_h(\cdot\mid x,u)$ is the controlled transition kernel, $c_h(x,u)$ is
the expected one-step cost, and $\gamma_h\in(0,1)$ is the one-step discount factor.
The kernel $P_h$ is induced by the projected Euler--Maruyama transition. Given $X_k=x$ and $U_k=u$, let
\begin{align}
    \widetilde X_{k+1}
    &=X_k-(U_k+\theta X_k)h+\sigma\sqrt{h}\,\varepsilon_{k+1},
    \label{eq:rl-euler-pre}\\
    \Delta L_{k+1}
    &=\bigl(-\widetilde X_{k+1}\bigr)^+,
    \label{eq:rl-euler-regulator}\\
    X_{k+1}
    &=\bigl(\widetilde X_{k+1}\bigr)^+=\widetilde{X}_{k+1}+\Delta L_{k+1},
    \label{eq:rl-euler-state}
\end{align}
where $\varepsilon_k$ with $k\geq 1$ are independent standard normal random
variables. The projection in \eqref{eq:rl-euler-regulator}--
\eqref{eq:rl-euler-state} enforces the reflecting boundary, and
$P_h(\cdot\mid x,u)$ is the conditional law of $X_{k+1}$. Each transition produces the realized cost and reward as
\begin{equation}
    \widehat c_{h,k}=C(U_k)h+p\Delta L_{k+1},
    \qquad
    R_k=-\widehat c_{h,k},
    \label{eq:rl-stage-cost}
\end{equation}
with $c_h(x,u)=\mathbb E[\widehat c_{h,k}\mid X_k=x,U_k=u]$. 
%Thus, the learning algorithms require only samples $(X_k,U_k,R_k,X_{k+1})$ and do not evaluate either $P_h$ or $c_h$ analytically. The cumulative regulator is not included as a separate state variable: for this stationary infinite-horizon problem, only its current increment enters the one-step cost.

The continuous-time discount is preserved under discretization by setting
\begin{equation}
    \gamma_h=e^{-\alpha h}.
    \label{eq:rl-discount}
\end{equation}
Indeed, at time $t_k=kh$,
we have $e^{-\alpha t_k}=(e^{-\alpha h})^k=\gamma_h^k$. Hence, for a stationary policy
$\pi$, the discrete objective
\begin{equation}
    J_h^{\pi}(x)
    =\mathbb E_x^{\pi}\!\left[
        \sum_{k=0}^{\infty}\gamma_h^k
        \bigl(C(U_k)h+p\Delta L_{k+1}\bigr)
      \right]
    \label{eq:rl-discrete-objective}
\end{equation}
is a direct approximation of the original discounted control objective. For simulations over $[0,T]$, with $T=Kh$, we write $J_{h,T}^{\pi}(x)$ for the expected sum restricted to $k=0,1,\ldots,K-1$. 

\subsection{Learning methods}

In this subsection, we compare Q-learning, PPO, SAC, and CEM with the HJB-based feedback. For simulation purposes only, the learning methods search over policies with control actions in $[0,u_{\max}]$, where $u_{\max}=10$. All four methods use random initialization. Their implementations and training configurations are described below.

Q-learning estimates a tabular action-value function and deploys the corresponding greedy policy \cite{watkins1992q}. Training consists of $20000$ Q-table updates. At each update, the Bellman targets are estimated using $20$ independently sampled one-step transitions for every state--action grid pair.

PPO uses a tabular categorical actor with $300$ state grid points and $40$ actions, together with a tabular critic and the clipped surrogate update of \cite{schulman2017proximal}. It is trained for $600$ iterations, each using $200$ independently simulated trajectories of $500$ time steps.

SAC uses a tanh-squashed Gaussian actor, twin critics, experience replay, target networks, and automatic entropy tuning, following the continuous-action formulation of \cite{haarnoja2018soft}. It is trained for $500$ iterations, each collecting $200$ independent trajectories of $100$ time steps and applying $100$ gradient updates.

CEM searches the parameterized deterministic feedback class
\begin{equation}
    u_{\vartheta}(x)
    =\operatorname{clip}\!\left(
        \vartheta_0+\vartheta_1x+\vartheta_2x^2
        +\vartheta_3\log(1+x),\,0,\,u_{\max}
      \right)
    \label{eq:cem-policy-class}
\end{equation}
using a Gaussian distribution over the parameter vectors \cite{rubinstein2004cross,pourchot2018cem}. At each of its $150$ iterations, $100$ candidate policies are sampled. Each candidate's cost is estimated from $50$ independent simulated trajectories. The parameter vectors of the $20$ lowest-cost policies are then used to update the mean and standard deviation of the search distribution.

The HJB control policy
\begin{equation}
    u^*(x)=F'\!\left(Q'(x)\right)
    \label{eq:hjb-feedback-benchmark}
\end{equation}
is computed separately and used only as an evaluation benchmark.

\subsection{Evaluation protocol and results}
We use the parameters $\sigma=1$, $\theta=0.5$, $\alpha=0.5$, and $p=1.1$,
together with the cost function in Example~\ref{example: exponential cost}. Policy and
representative trajectories are generated over $T=2$ using $10000$
Euler--Maruyama steps. To approximate the infinite-horizon objective, policy
costs are evaluated over $T=20$, using $10000$ steps.  All policies are evaluated on the same $500$ Brownian sample paths, which allows the costs to be compared under the same random disturbances. For PPO and SAC, actions are sampled from the learned policy
during cost evaluation. Their policy profiles display the conditional mean
control
\begin{equation}
    \bar u_{\pi}(x)=\mathbb E_{\pi}[U\mid X=x],
    \label{eq:conditional-mean-policy}
\end{equation}
which specifies the action distribution at each state.

For each policy, we compute the discounted cost over $[0, T]$ in every simulation and average these costs to estimate $J_{h, T}^\pi(x)$. The column of Average cost reports this estimate. The column Difference from HJB gives the policy's average cost minus the average cost of the HJB reference. A positive value indicates a higher simulated cost than the HJB reference. To compute the standard error, we subtract the HJB cost from the learned-policy cost in each simulation, calculate the sample standard deviation of these differences, and divide by the square root of the number of simulations. The resulting value measures the Monte Carlo uncertainty in the estimated cost difference for the fixed trained policies. The results are presented in Table~\ref{tab:paired-performance}.
\begin{table}[htbp]
    \centering
    \caption{Long-horizon performance over 500 common Brownian paths
    ($T=20$, 10000 Euler--Maruyama steps).}
    \label{tab:paired-performance}
    \begin{tabular}{lrrr}
        \hline
        Method & Average cost & Difference from HJB & Standard error \\
        \hline
        HJB & 4.051149 & 0 & -- \\
        CEM & 4.051330 & 0.000181 & 0.000084 \\
        Q-learning & 4.051934 & 0.000786 & 0.000192 \\
        SAC & 4.072596 & 0.021447 & 0.000318 \\
        PPO & 4.113396 & 0.062248 & 0.000495 \\
        \hline
    \end{tabular}
\end{table}

\begin{figure}[htbp]
    \centering
    \includegraphics[width=\linewidth]{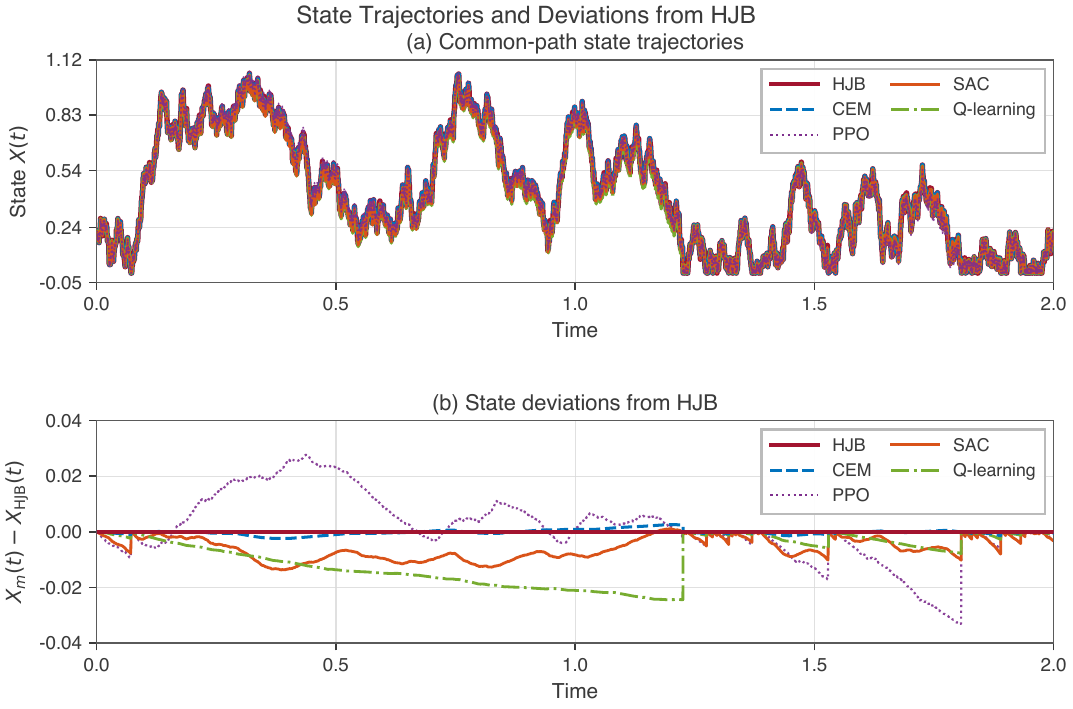}
    \caption{Simulated state trajectories under the learned policies and the HJB feedback. Panel (a)
    shows trajectories generated using the same Brownian increments. Panel (b) shows the difference between each trajectory and the HJB trajectory.}
    \label{fig: comparison of optimal u}
\end{figure}

\begin{figure}
    \centering
    \includegraphics[width=\linewidth]{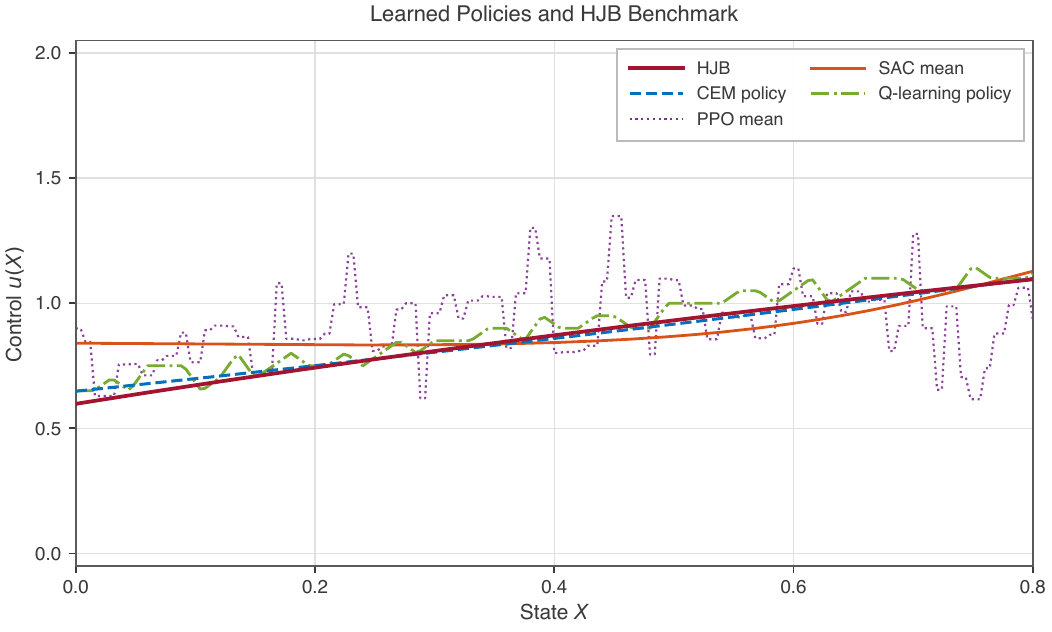}
    \caption{Learned feedback policies and the numerical HJB benchmark. The PPO and SAC
curves show the mean action at each state. The CEM, Q-learning, and HJB curves represent
deterministic feedback controls.}
    \label{fig:HJB_RL_state_control}
\end{figure}

CEM achieves the closest average cost to the HJB reference, with a difference of $1.81\times10^{-4}$. Q-learning also attains a small cost difference of $7.86\times10^{-4}$. Among the two stochastic actor--critic implementations, SAC achieves a lower average cost than PPO.

Figure~\ref{fig: comparison of optimal u} compares the learned feedback policies with the HJB policy and displays the state differences. Figure~\ref{fig:HJB_RL_state_control} compares feedback profiles over the state intervals, where CEM and Q-learning track HJB policy closely, and SAC has a smoother mean-action policy than PPO. Together with the cost estimates, these comparisons provide a numerical assessment of sample-based policy approximation for the DCP under our model and training settings and motivate further approximation and transfer guarantees for
learned controls.

%%%%%%%%%%%%%%%%%%%%%%%%%%%%%%%%%%%%%%%%%%%%%%%

\section{Conclusions}\label{sec conclusion}

% The waiting time and queue length are two of the most important characteristics of a queueing system. 
% In many applications, such as cloud computing platforms, companies offer customers an interface displaying waiting times, as queue lengths alone may not provide sufficient information. Customers, for example, are typically more concerned with how long they have to wait rather than the number of people ahead of them.
% This paper contributes to the literature on drift-rate control problems by introducing a state-dependent intensity in the discrete-view waiting time process and a non-trivial arrival-intensity control cost function for diffusion models. We tackle a stochastic control problem in a single-server queueing system, where the state process represents the offered waiting time. The arrival-intensity control problem includes a state-dependent intensity and generally distributed patience times. The use of control and its associated cost function is complex due to the state-dependent nature of the process. However, the physical interpretation of the control cost ensures the feasibility of our assumptions on the cost function.
% Furthermore, we explore the application of the REINFORCE algorithm for solving stochastic control problems, drawing inspiration from recent literature. We compare the advantages and limitations of conventional simulation from theoretical results and RL algorithms, which sets the stage and provides motivation for future research aimed at improving stochastic control methods in queueing models. 

Waiting time and queue length are two of the most important characteristics of a queueing system. 
In applications such as cloud computing platforms, waiting time is often displayed to users, as it provides more actionable information than queue length alone. Customers typically care more about how long they must wait than how many are ahead of them.
This paper develops an arrival-intensity control framework for the offered waiting-time process of a single-server queue by introducing a diffusion-scale control perturbation of the state-dependent arrival intensity. 
We address a stochastic control problem in a single-server queue, where the state process represents the offered waiting time. The model incorporates state-dependent intensity and generally distributed patience times, making the control and its cost function analytically challenging. Nonetheless, the physical interpretation of the operational running cost supports the feasibility of our assumptions.
% {\color{blue}
We also evaluated reinforcement learning methods for constructing
feedback policies from simulated state transitions and realized costs.
Q-learning and CEM achieved average
discounted costs close to those of the independently computed HJB
reference, with CEM yielding the smallest cost difference. SAC
achieved a lower average cost than PPO among the two stochastic
actor--critic implementations. These results support the effectiveness
of sample-based learning for approximating the discounted performance
of the HJB feedback in the proposed diffusion control problem.
The variation in performance across the tested implementations
motivates further investigation of policy representations and
training procedures, together with theoretical guarantees for
the resulting learned controls.
% }

% The waiting time models presented in this paper are intriguing, and we aim to extend the current model to address more realistic control problems. 
% Specifically, we plan to incorporate non-linear and time-state-dependent stochastic systems, focusing on queuing systems of particular interest. One potential direction is to formulate a partial differential equation form of the HJB equation, incorporating the time and position, whose solution is likely to present significant challenges.
% Additionally, we intend to explore the extension to multi-class systems, aiming to formulate a stochastic control problem based on the waiting-time state process. This extension could provide further insights into complex queuing scenarios involving multiple classes of customers with varying characteristics.
% Moreover, we recognize the importance of developing stable and flexible algorithms for solving stochastic control problems in queuing theory. Investigating these computational approaches will be highly beneficial for advancing the practical application of stochastic control methods in real-world queueing systems.

The waiting time control models developed in this paper open avenues for addressing more realistic control problems. We plan to extend the current framework to nonlinear, time- and state-dependent stochastic systems, with a focus on queueing models of practical relevance. One promising direction is to formulate a partial differential equation version of the HJB equation that incorporates both time and state variables, which poses significant analytical challenges.
We also aim to extend the model to multi-class matching systems (cf. \cite{xie2024multi}), formulating a stochastic control problem based on the waiting-time state process. This would offer deeper insights into complex queueing scenarios involving heterogeneous customer classes.
In parallel, we recognize the need for stable and flexible RL algorithms to solve these stochastic control problems. 
Investigating computational methods, particularly RL, will be essential for advancing the practical application of stochastic control in real-world queueing systems.

\backmatter

% \bmhead{Supplementary information}

% If your article has accompanying supplementary file/s please state so here. 

% Authors reporting data from electrophoretic gels and blots should supply the full unprocessed scans for key as part of their Supplementary information. This may be requested by the editorial team/s if it is missing.

% Please refer to Journal-level guidance for any specific requirements.

\bmhead{Acknowledgments}

The authors would like to acknowledge Dr. Ananda Weerasinghe for his guidance, patience, enthusiasm, and inspiration throughout the research and the writing of the paper.

% Acknowledgments are not compulsory. Where included they should be brief. Grant or contribution numbers may be acknowledged.

% Please refer to Journal-level guidance for any specific requirements.

% \section*{Declarations}

% Some journals require declarations to be submitted in a standardised format. Please check the Instructions for Authors of the journal to which you are submitting to see if you need to complete this section. If yes, your manuscript must contain the following sections under the heading `Declarations':

% \begin{itemize}
% \item Conflict of interest/Competing interests (check journal-specific guidelines for which heading to use)
% \item Ethics approval 
% \item Consent to participate
% \item Consent for publication
% \item Availability of data and materials
% \item Authors' contributions
% \end{itemize}

% \noindent
% If any of the sections are not relevant to your manuscript, please include the heading and write `Not applicable' for that section. 

% %%===================================================%%
% %% For presentation purpose, we have included        %%
% %% \bigskip command. please ignore this.             %%
% %%===================================================%%
% \bigskip

% Springer journals and proceedings: \url{https://www.springer.com/gp/editorial-policies}

% Nature Portfolio journals: \url{https://www.nature.com/nature-research/editorial-policies}

% \textit{Scientific Reports}: \url{https://www.nature.com/srep/journal-policies/editorial-policies}

% BMC journals: \url{https://www.biomedcentral.com/getpublished/editorial-policies}
% \end{flushleft}

\begin{appendices}

%%%%%%%%%%%%%%%%%%%%%%%%%%%%%%%%%%
\section{Proof of Theorem \ref{A solution to the HJB equation}}
\label{Appendix B} 

We prove Theorem~\ref{A solution to the HJB equation} by analyzing the nonlinear ordinary differential equation \eqref{rewritten HJB}. For a clear exposition, we relabel it as 
\begin{equation}
    \frac{\sigma^2}{2}Q''(x)
    -F(Q'(x))
    -\theta xQ'(x)
    -\alpha Q(x)
    =0,
    \quad  x>0,
    \label{eq: appendix HJB}
\end{equation}
subject to
\[
    Q'(0)=-p.
\]
Since \eqref{eq: appendix HJB} is a second-order equation while only one boundary condition is prescribed at the origin, we introduce a parameterized family of initial-value problems by specifying the initial curvature
\[
    Q''(0)=r,\quad  r\geq0.
\]
We then identify a critical parameter for which the resulting solution is defined on $[0,\infty)$, remains strictly decreasing and convex, and is bounded. 

To emphasize the dependence of the artificial parameter $r$, we consider a family of solutions $\{Y_r\}_{r\geq 0}$ to the following differential equation parameterized by $r\geq0$: 
\begin{equation}
% \begin{aligned}
    \frac{\sigma^2}{2} Y_r''(x) - F(Y_r'(x)) - \theta xY_r'(x) - \alpha Y_r(x) = 0, \quad
    Y_r'(0) = -p, \quad Y_r''(0) = r. 
    \label{yrequation}
% \end{aligned}
\end{equation}
Evaluating \eqref{yrequation} at $x=0$ determines the remaining initial value:
\begin{equation}
    Y_r(0)
    =K_r
    :=
    \frac1\alpha
    \left(
        \frac{\sigma^2}{2}r-F(-p)
    \right).
    \label{eq: Kr}
\end{equation}
Here, we employ the notation $K_r$ to exhibit the dependence of the constant $K$ on the parameter $r$. 
Now, our objective is to find a suitable $r\geq 0$ so that 
\begin{enumerate}[(i)]
    \item There exists a solution $Y_r(\cdot)$ to \eqref{yrequation} defined on $[0, \infty)$, and
    
    \item $Y_r'(x)$ is bounded and $Y_r'(x)<0$ for all $x\geq0$. 
\end{enumerate}
We aim to characterize the solution profile and its dependence on the parameter $r$ as $x$ varies.

% To simplify \eqref{yrequation}, 
For ease of analysis, we define the reduction of order $W_r:= Y_r'$, which implies
\begin{equation}
    \frac{\sigma^2}{2} W_r'(x) - F(W_r(x)) - \theta x W_r(x) - \alpha \int_0^x W_r(s)ds-\alpha K_r = 0, 
    \label{wr with integral}
\end{equation}
where $K_r$ is a constant obtained in \eqref{eq: Kr} and the initial conditions are $W'_r(0) = r$ and $W_r(0) = -p$. 
% demonstrate that this constant $K$ depends on parameter $r$. 
If we differentiate \eqref{wr with integral}, we can obtain a new second-order equation:
\begin{equation}
\begin{aligned}
    \frac{\sigma^2}{2} W_r''(x) - F'(W_r(x))W_r'(x)-\theta x W_r'(x) - (\theta+\alpha)W_r(x) = 0, \\
    W_r(0) = -p, \quad W_r'(0) = r. 
    \label{wrequation}
\end{aligned}
\end{equation}
We observe that $F'$ is Lipschitz continuous on $(-\infty, -\delta]$ for any $\delta>0$ as described in \eqref{F'}. 
The initial value problem (IVP) has a unique solution as long as $W_r$ remains strictly negative. 
Let
\[
    c_r
    :=
    \sup\left\{
        x>0:
        W_r(s)<0 \text{ for all } s\in[0,x)
    \right\}.
\]
Hence, our objective is to find a suitable $r\geq 0$ so that $c_r=+\infty$ and the corresponding $W_r(\cdot)$ defined on $[0, \infty)$ is bounded and strictly negative for all $x\geq0$. 
Thus, $c_r$ denotes the maximal interval on which the solution remains strictly negative; when $c_r<\infty$ and no explosion occurs before $c_r$, it represents the first time at which $W_r$ reaches zero. 
% Notice that the definition of $c_r$ could be interpreted as the first reaching time of the origin for solution $W_r$ or up to which $W_r$ is well defined. 
% We intend to interpret $c_r$ using the former explanation, and the explosion time will be labeled by the notation $e_r$ if necessary. 

The original transform $F$ \eqref{Legendre transform in our case C} is finite only on $(-\infty, 0]$. 
Since trial solutions in the parameterized family may reach or cross zero and to avoid the sign condition in \eqref{rewritten HJB} in our objective (ii), we need to introduce a linear extension of the Legendre transform $F$. 
Fix $0 < \delta < \min\{p, -C'(0)\}$, and define
\begin{equation}
    \tilde F(y)
    :=
    \begin{cases}
        F(y), & y\leq-\delta,\\
        F(-\delta)+F'(-\delta)(y+\delta),
        & y>-\delta.
    \end{cases}
    \label{Tilde F}
\end{equation}
Then $\tilde F\in C^1(\mathbb R)$ and $\tilde F'$ is bounded and nondecreasing, since $F'$ is nondecreasing on $(-\infty,0)$ and the extension has constant slope $F'(-\delta)$ on $(-\delta,\infty)$.
% Then $\tilde F\in C^1(\mathbb R)$ and $\tilde F'$ is bounded and nondecreasing. 
In particular, the corresponding extended initial-value problem is locally well posed.

Now, to exhibit the existence and uniqueness of a solution to \eqref{wr with integral}, we consider the equation \eqref{wrequation} with $F(\cdot)$ replaced by $\tilde{F}(\cdot)$ such that 
\begin{equation}
\begin{aligned}
    \frac{\sigma^2}{2} \tilde{W}_r''(x) = \tilde{F}'(\tilde{W}_r(x))\tilde{W}_r'(x)
    +\theta x \tilde{W}_r'(x) + (\theta+\alpha)\tilde{W}_r(x), \\
    \tilde{W}_r(0) = -p, \quad \tilde{W}_r'(0) = r. 
    \label{wrequation with tilde F}
\end{aligned}
\end{equation}
In this case, it is not necessary to check the sign condition consistently. 
Integrating both sides of \eqref{wrequation with tilde F} yields
\begin{equation*}
    \tilde{W}_r(x) = \frac{2}{\sigma^2} \int_0^x \left[\tilde{F}(\tilde{W}_r(s)) + \theta s \tilde{W}_r(s) + \alpha K   \right]\,ds + \frac{2\alpha}{\sigma^2} \int_0^x \int_0^s \tilde{W}_r(v)\,dvds. 
\end{equation*}
Since $\tilde F'$ is bounded and locally Lipschitz, the right-hand side of \eqref{wrequation with tilde F} is locally Lipschitz in $(\tilde W_r,\tilde W_r')$. Hence, standard ODE theory yields a unique local solution. 
% Since the extension of the Legendre transform $\tilde{F}(\cdot)$ is Lipschitz continuous, there exists a unique solution to \eqref{wrequation with tilde F} in the neighborhood of the origin. 

It is worth mentioning that for fixed $r\geq0$, the auxiliary solution depends on the extension parameter $\delta$. 
For $0<\delta_1<\delta_2$, we consider $\tilde{W}_{r, \delta_1}$ and $\tilde{W}_{r, \delta_2}$ corresponding to $\delta_1$ and $\delta_2$, respectively. Let $\eta = \inf\{x\in[0, c_r): \tilde{W}_{r, \delta_1}(x)\vee\tilde{W}_{r, \delta_2}(x)=-\delta_2\}$. 
We conclude that $\tilde{W}_{r, \delta_1}(x)\equiv\tilde{W}_{r, \delta_2}(x)$ for all $x\in[0, \eta)$ when $\delta_2$ is arbitrarily small since the uniqueness. 
For brevity, we omit $\delta$ in the notation of $\tilde{W}_{r, \delta}(\cdot)$ when there is no ambiguity.
Notice that $W_r(x)$ agrees with $\tilde{W}_r(x)$ for all $x\in[0, c_r)$ when $\delta$ is arbitrarily small since $\tilde{F}(x)$ coincides with $F(x)$ when $x\in(-\infty, -\delta]$ and by the uniqueness. 
Therefore, it is trivial that there also exists a unique solution to \eqref{wrequation} for $x\in[0, c_r)$ by taking an arbitrarily small $\delta$. 
Hence, we conclude that the corresponding $Y_r$ is the solution to \eqref{yrequation}, which is equivalent to \eqref{wrequation}. 

We next study how the solution profile depends on the curvature parameter $r$. The key tool is a comparison result for the family $\{W_r\}_{r\ge0}$.

\begin{lemma}[Comparison lemma]
\label{comparison lemma}
Let $0<r_2<r_1$. Then, on
$[0,c_{r_1}\wedge c_{r_2})$, the following statements hold:
\begin{enumerate}[(i)]
    \item For the solutions of the extended equation \eqref{wrequation with tilde F}, we have $\tilde W_{r_2}'(x) < \tilde W_{r_1}'(x)$ and $\tilde W_{r_2}(x) \leq \tilde W_{r_1}(x)$, where equality in the second inequality holds only at $x=0$. Moreover, $\tilde Y_{r_2}(x) < \tilde Y_{r_1}(x)$.

    \item Consequently, for the solutions of the original equation \eqref{wr with integral}, we have $W_{r_2}'(x)<W_{r_1}'(x)$ and $W_{r_2}(x)\leq W_{r_1}(x)$, where equality in the second inequality holds only at $x=0$. Moreover, $Y_{r_2}(x)<Y_{r_1}(x)$. 
\end{enumerate}
\end{lemma}

\begin{proof}
We first prove part (i). Since $\tilde W_{r_1}(0) = \tilde W_{r_2}(0) = -p$ and $\tilde W_{r_1}'(0) = r_1 > r_2 = \tilde W_{r_2}'(0)$, we have $\tilde W_{r_1}'(x) >\tilde W_{r_2}'(x)$ for all sufficiently small $x>0$.
Suppose, to the contrary, that the derivative ordering fails for some $x >0$. 
Let
\[
    a
    :=
    \inf\left\{
        x>0:
        \tilde W_{r_1}'(x)
        =
        \tilde W_{r_2}'(x)
    \right\}.
\]
Then we have $\tilde W_{r_1}'(x) > \tilde W_{r_2}'(x)$ for $0\leq x<a$, and $\tilde W_{r_1}'(a) = \tilde W_{r_2}'(a)$. 
Consequently, 
\[
    \tilde W_{r_1}(x)
    -
    \tilde W_{r_2}(x)
    =
    \int_0^x
    \left(
        \tilde W_{r_1}'(s)
        -
        \tilde W_{r_2}'(s)
    \right)ds
    >0
\]
for every $x\in(0,a]$.

Recall that similar to \eqref{wr with integral}, the extended equation \eqref{wrequation with tilde F} may be written in integral form as 
\begin{equation}
\frac{\sigma^2}{2}\tilde W_r'(x)
=
\tilde F(\tilde W_r(x))
+\theta x\tilde W_r(x)
+\alpha\int_0^x\tilde W_r(s)\,ds
+\alpha K_r,
\label{eq: extended integral comparison}
\end{equation}
where, since $\delta<p$, the integrating constant becomes
\[
    K_r
    =
    \frac{1}{\alpha}
    \left(
        \frac{\sigma^2}{2}r-F(-p)
    \right)
    =
    \frac{1}{\alpha}
    \left(
        \frac{\sigma^2}{2}r-\tilde F(-p)
    \right).
\]
In particular, $\alpha(K_{r_1}-K_{r_2}) = \frac{\sigma^2}{2}(r_1-r_2)>0$. 
Subtracting \eqref{eq: extended integral comparison} for
$r_2$ from that for $r_1$ and evaluating at $x=a$ gives
\begin{align*}
0
&=
\tilde F(\tilde W_{r_1}(a))
-
\tilde F(\tilde W_{r_2}(a))
+\theta a
\left(
    \tilde W_{r_1}(a)
    -
    \tilde W_{r_2}(a)
\right)\\
&\quad 
+\alpha
\int_0^a
\left(
    \tilde W_{r_1}(s)
    -
    \tilde W_{r_2}(s)
\right)ds
+\alpha(K_{r_1}-K_{r_2}).
\end{align*}
Since $\tilde F'\geq0$, the function $\tilde F$ is
nondecreasing. Hence, it is natural that $\tilde F(\tilde W_{r_1}(a)) - \tilde F(\tilde W_{r_2}(a)) \geq0$.
Furthermore, we observe that 
\[
\tilde W_{r_1}(a) - \tilde W_{r_2}(a)>0, 
\quad \int_0^a \left(\tilde W_{r_1}(s) - \tilde W_{r_2}(s) \right)ds>0, 
\]
and $K_{r_1}-K_{r_2}>0$. 
Thus, the right-hand side is strictly positive, which is a contradiction.
Therefore, $\tilde W_{r_2}'(x) < \tilde W_{r_1}'(x)$ throughout $[0,c_{r_1}\wedge c_{r_2})$.
Integrating the derivative inequality and in conjunction with $\tilde W_{r_1}(0)=\tilde W_{r_2}(0)=-p$ gives
\[
    \tilde W_{r_2}(x)
    <
    \tilde W_{r_1}(x),
    \quad  x>0,
\]
with equality only at $x=0$.

Finally, by the reduction of order substitution, we have 
\[
\begin{aligned}
    \tilde Y_{r_1}(x)
    -
    \tilde Y_{r_2}(x)
    &=
    K_{r_1}-K_{r_2}
    +
    \int_0^x
    \left(
        \tilde W_{r_1}(s)
        -
        \tilde W_{r_2}(s)
    \right)ds
    >0,
\end{aligned}
\]
which proves part (i).

For part (ii), fix any $x<c_{r_1}\wedge c_{r_2}$.
Since $W_{r_i}$ is continuous and strictly negative on the compact interval $[0,x]$, there exists $\delta>0$ sufficiently small such that $W_{r_i}(s)\leq-\delta$ for $0\leq s\leq x$ and $i=1,2$. 
On this interval, we have $\tilde F(W_{r_i}(s)) = F(W_{r_i}(s))$. 
Hence, by uniqueness of the corresponding initial-value problems, it is straightforward that $W_{r_i}(s) = \tilde W_{r_i}(s)$ for $0\leq s\leq x$ and $i=1,2$.
The comparison results established in part (i) therefore yield
\[
    W_{r_2}'(x)<W_{r_1}'(x),
    \quad 
    W_{r_2}(x)<W_{r_1}(x)
\]
for every $x\in(0,c_{r_1}\wedge c_{r_2})$, with equality
$W_{r_1}(0)=W_{r_2}(0)=-p$ at the origin. The same argument gives $Y_{r_2}(x)<Y_{r_1}(x)$. 
This completes the proof.
\end{proof}

We next exmaine the qualitative behavior of the auxiliary solution $\tilde{W_r}$ for a fixed $r\geq 0$. 
Suppose that $\tilde W_r'(\xi)=0$ at some $\xi\geq0$. Evaluating \eqref{wrequation with tilde F} at $x=\xi$ gives
\begin{equation}
    \frac{\sigma^2}{2}\tilde W_r''(\xi)
    =
    (\theta+\alpha)\tilde W_r(\xi).
    \label{eq: critical point Wr}
\end{equation}
Since $\theta+\alpha>0$, the sign of $\tilde W_r''(\xi)$ is determined by the sign of $\tilde W_r(\xi)$.
In particular, the following possibilities arise:
\begin{enumerate}[(i)]
    \item The case $\tilde W_r(\xi)=\tilde W_r'(\xi)=0$ cannot occur. 
    Indeed, \eqref{eq: critical point Wr} then gives $\tilde W_r''(\xi)=0$. By uniqueness of the IVP \eqref{wrequation with tilde F} with initial data $\tilde W_r(\xi)=\tilde W_r'(\xi)=0$, the corresponding solution would be identically zero, contradicting $\tilde W_r(0)=-p<0$.
    
    \item If $\tilde W_r(\xi)>0$, then $\tilde W_r''(\xi)>0$, and hence $\xi$ is a strict local minimum. 

    \item If $\tilde W_r(\xi)<0$, then $\tilde W_r''(\xi)<0$, and hence $\xi$ is a strict local  maximum.
\end{enumerate}

Consequently, $\tilde W_r$ has no positive local maximum and no negative local minimum. In particular, on any interval over which $\tilde W_r$ remains strictly negative, every critical point must be a strict local maximum.

The same conclusion applies to the original solution $W_r$ on $[0,c_r)$, since $W_r(x)<0$ there and the auxiliary and original equations coincide on compact subintervals of $[0,c_r)$. Thus, any critical point of $W_r$ in $[0,c_r)$ is necessarily a strict negative local maximum. It remains possible, however, that $W_r$ approaches zero at the endpoint $c_r$. The following lemma shows that such a contact cannot occur tangentially, namely, if $W_r$ reaches zero, its derivative must be strictly positive at the crossing.

\begin{lemma}
\label{touch x axis tangentially}
Let $W_r$ be a solution to \eqref{wr with integral} satisfying $W_r(0)=-p$ and $W_r'(0)=r\geq0$. 
Suppose that $c_r<\infty$ and $\lim_{x\to c_r^-}W_r(x)=0$. Then $\lim_{x\to c_r^-}W_r'(x)>0$. In particular, $W_r$ cannot reach zero tangentially.
\end{lemma}

\begin{proof}
Since $W_r(x)<0$ for $x\in[0,c_r)$ and $W_r(x)\to0$ as $x\to c_r^-$, equation
\eqref{wr with integral} gives
\[
    \frac{\sigma^2}{2}W_r'(x)
    =
    F(W_r(x))
    +\theta xW_r(x)
    +\alpha\int_0^x W_r(s)\,ds
    +\alpha K_r.
\]
Since $c_r<\infty$, $W_r$ is bounded on $[0,c_r)$.
Moreover, $F(W_r(x))\to F(0)=0$ and $xW_r(x)\to0$ as $x\to c_r^-$. Hence the limit $\ell_r:=\lim_{x\to c_r^-}W_r'(x)$ exists and satisfies
\begin{equation}
    \frac{\sigma^2}{2}\ell_r
    =
    \alpha
    \left(
        \int_0^{c_r}W_r(s)\,ds+K_r
    \right).
    \label{eq: derivative limit cr}
\end{equation}

We first observe that $\ell_r\geq0$. Indeed, if $\ell_r<0$, then $W_r'(x)<0$ for all $x$ sufficiently close to $c_r$. Thus $W_r$ would be strictly decreasing near $c_r$, which is incompatible with $W_r(x)<0$ and $\lim_{x\to c_r^-}W_r(x)=0$. 

It therefore remains to exclude $\ell_r=0$. Suppose that $\lim_{x\to c_r^-}W_r'(x)=0$. 
Then \eqref{eq: derivative limit cr} implies
\begin{equation}
    \int_0^{c_r}W_r(s)\,ds+K_r=0.
    \label{eq: tangential condition}
\end{equation}
Using \eqref{eq: tangential condition} in \eqref{wr with integral}, for $x<c_r$ we obtain
\begin{equation}
\begin{aligned}
    \frac{\sigma^2}{2}W_r'(x)
    &=
    F(W_r(x))
    +\theta xW_r(x)
    -\alpha\int_x^{c_r}W_r(s)\,ds.
\end{aligned}
\label{eq: Wr near cr}
\end{equation}
Dividing by $W_r(x)<0$ implies
\begin{equation}
\begin{aligned}
    \frac{\sigma^2}{2}
    \frac{W_r'(x)}{W_r(x)}
    &=
    \frac{F(W_r(x))}{W_r(x)}
    +\theta x
    +\alpha
    \frac{\int_x^{c_r}W_r(s)\,ds}{-W_r(x)}.
\end{aligned}
\label{eq: logarithmic derivative Wr}
\end{equation}

We next examine the first term on the right-hand side. For $y\in(C'(0),0)$, let $u(y)>0$ denote the unique maximizer in the definition of $F$ \eqref{Legendre transform}. Then we have $C'(u(y))=y$ and $F(y)=yu(y)-C(u(y))$.
Consequently, $\frac{F(y)}{y} = u(y)-\frac{C(u(y))}{y} \geq u(y)$, because $C(u(y))\geq0$ and $y<0$. 
Since $C$ is strictly decreasing, $C''(u) > 0$, and $\lim_{u\to\infty}C(u)=0$, we have $C'(u)\to 0^-$ as $u\to\infty$, and therefore $u(y)\to\infty$ as $y\to0$. 
It follows that
\begin{equation}
    \lim_{y\to0^-}\frac{F(y)}{y}=+\infty.
    \label{eq: F over y infinity}
\end{equation}

On the other hand, since $W_r(x)\to0$ as $x\to c_r^-$, the last term in \eqref{eq: logarithmic derivative Wr} does not offset this divergence. More precisely, for $x$ sufficiently close to $c_r$, the mean value theorem gives
\[
    \int_x^{c_r}W_r(s)\,ds
    =
    (c_r-x)W_r(\xi_x)
\]
for some $\xi_x\in(x,c_r)$. 
Under the assumption $\ell_r=0$, $W_r$ must be nondecreasing on a sufficiently small left neighborhood of $c_r$; otherwise a sequence of negative derivatives bounded away from zero would contradict $\ell_r=0$. Hence, we have $W_r(x)\leq W_r(\xi_x)<0$ for $x$ sufficiently close to $c_r$, and therefore
\[
    -(c_r-x)
    \leq
    \frac{\int_x^{c_r}W_r(s)\,ds}{-W_r(x)}
    \leq0.
\]
Since $\theta x\geq0$ and combining the preceding estimate with \eqref{eq: logarithmic derivative Wr}, we obtain
\begin{equation}
    \frac{\sigma^2}{2}
    \frac{W_r'(x)}{W_r(x)}
    \geq
    \frac{F(W_r(x))}{W_r(x)}
    -\alpha(c_r-x),
    \quad  x<c_r. 
    \label{eq: lower bound logarithmic derivative}
\end{equation}
This further yields $\lim_{x\to c_r^-}\frac{W_r'(x)}{W_r(x)} = +\infty$. 
That is, for any $M>0$, there exists some $\delta_0$ such that $\frac{W_r'(c_r -\delta)}{W_r(c_r - \delta)} > M$ whenever $\lvert\delta\rvert < \delta_0$.
Hence, $- W_r'(c_r - \delta) > -M W_r(c_r - \delta)$. 
Since $-W_r'(c_r - \delta) < 0$, we have $-M W_r(c_r - \delta) < 0$, which contradicts $W_r(c_r - \delta) < 0$. 
In a nutshell, since $W_r(x)<0$, this implies that $W_r'(x)<0$ for all $x$ sufficiently close to $c_r$. Thus $W_r$ is strictly decreasing on a left neighborhood of $c_r$, which contradicts the assumption of the first reaching point. 
Hence $\ell_r\neq0$. Since we already established $\ell_r\geq0$, we conclude that $\lim_{x\to c_r^-}W_r'(x)>0$. 
This completes the proof. 
\end{proof}

% We have described the properties of the solution profiles of $W_r(x)$ with respect to different $r$ and $x$ in Lemma \ref{comparison lemma} and Lemma \ref{touch x axis tangentially}. 
% Next, we would like to understand the solution profiles of $W_r$ when $r$ is small and when $r$ is large. 
% Moreover, these results could help us to visualize the behaviors of solutions $W_r$'s with respect to different $r\geq0$. 

Lemmas \ref{comparison lemma} and \ref{touch x axis tangentially} establish key qualitative properties of the solution family $\{W_r\}_{r\geq0}$, including its monotone dependence on the parameter $r$ and the behavior of a solution upon reaching zero. 
We next examine the solution profiles corresponding to small and large values of $r$. These two regimes will be used to identify a critical parameter $r^*$ for which the corresponding solution remains strictly negative on $[0,\infty)$ and satisfies the desired qualitative properties.

\begin{lemma}
Consider the sequences of solutions $\{\tilde{W}_r(x)\}_{r\geq 0}$ and $\{W_r(x)\}_{r\geq0}$ to \eqref{wrequation with tilde F} and \eqref{wrequation} for all $x\in[0, c_r)$, the following conclusions hold:
\begin{enumerate}[(i)]
    \item There exists a $r_0>0$ small enough such that for each $0<r\leq r_0$, we have $\tilde{W}_r(x)<0$ for all $x\in[0, c_r)$ and it has a local maximum.
    
    \item The same conclusion holds for $W_r$. 
\end{enumerate}
\label{small r case}
\end{lemma}

\begin{proof}
We first consider the auxiliary equation \eqref{wrequation with tilde F} with $r=0$. Since $\tilde W_0(0)=-p<0$ and $\tilde W_0'(0)=0$, evaluating \eqref{wrequation with tilde F} at $x=0$ gives
\[
    \frac{\sigma^2}{2}\tilde W_0''(0)
    =
    (\theta+\alpha)\tilde W_0(0)
    =
    -(\theta+\alpha)p<0.
\]
Hence, there exists $\varepsilon>0$ sufficiently small such that $\tilde W_0'(x)<0$ for $x\in(0,\varepsilon]$, and, by decreasing $\varepsilon$ if necessary, $\tilde W_0(x)<0$ for $x\in[0,\varepsilon]$.

By the continuously dependence of the solutions of \eqref{wrequation with tilde F} on the initial data (see Chapter V in \cite{hartman2002ordinary} or Chapter 2 in \cite{teschl2012ordinary}), there exists $r_0>0$ such that $\tilde W_{r_0}(\varepsilon)<0$ and $\tilde W_{r_0}'(\varepsilon)<0$. 
On the other hand, we have $\tilde W_{r_0}'(0)={r_0}>0$. Thus, by continuity of $\tilde W_r'$, there exists $a_{r_0}\in(0,\varepsilon)$ such that $\tilde W_{r_0}'(a_{r_0})=0$. Moreover, for $r_0$ sufficiently small, we have $\tilde W_{r_0}(a_{r_0})<0$. 
Evaluating \eqref{wrequation with tilde F} at $a_{r_0}$ yields
\[
    \frac{\sigma^2}{2}\tilde W_{r_0}''(a_{r_0})
    =
    (\theta+\alpha)\tilde W_r(a_{r_0})<0.
\]
Hence $a_{r_0}$ is a strict negative local maximum of $\tilde W_{r_0}$. 
The comparison lemma (Lemma \ref{comparison lemma}) further deduces that for any $0 < r < r_0$, $\tilde W_r'(a_{r_0}) <\tilde W_{r_0}'(a_{r_0}) =0$, and there exists $a_r\in(0,a_{r_0})$ such that $\tilde W_r'(a_r)=0$.
Furthermore, the comparison lemma gives $\tilde W_r(a_r) < \tilde W_{r_0}(a_r)<0$.
Therefore, we have $\frac{\sigma^2}{2}\tilde W_r''(a_r) = (\theta+\alpha)\tilde W_r(a_r)<0$, and hence $a_r$ is a strict negative local maximum of $\tilde W_r$. 

Finally, since $\tilde W_r$ remains strictly negative on $[0,a_r]$, we may choose the extension parameter $\delta>0$ sufficiently small so that $\tilde F(\tilde W_r(x)) = F(\tilde W_r(x))$ whenever $x\in[0,a_r]$. By uniqueness of the corresponding initial-value problem, it is straightforward that $W_r(x)=\tilde W_r(x)$ for $x\in[0,a_r]$. 
Consequently, we have $W_r'(a_r)=0$, $W_r(a_r)<0$, and $W_r''(a_r)<0$, so $W_r$ also has a strict negative local maximum at $a_r$. Since $W_r<0$ on $[0,a_r]$, we have $a_r<c_r$. 
This completes the proof.
\end{proof}

\begin{lemma}
There exists $r_1>0$ such that, for every $r>r_1$, the auxiliary solution $\tilde W_r$ reaches zero in finite time. Consequently, $c_r<\infty$ and $\lim_{x\to c_r^-}W_r(x)=0$ for every $r>r_1$.
\label{large r case}
\end{lemma}

\begin{proof}
We first consider the auxiliary solution $\tilde W_r$ and show that it is possible for $\tilde W_r$ to cross the $x$-axis in finite time.  
Since $\tilde W_r(0)=-p$ and $\tilde W_r'(0)=r>0$, the solution initially moves upward from $-p$. Define
\[
    \tilde\tau_r
    :=
    \inf\left\{
        x>0:
        \tilde W_r(x)\notin(-p,0)
    \right\}.
\]
For $x\in(0,\tilde\tau_r)$, we have $-p<\tilde W_r(x)<0$. 

Similar to \eqref{wr with integral}, the integral form of the auxiliary equation is
\begin{equation}
    \frac{\sigma^2}{2}\tilde W_r'(x)
    =
    \tilde F(\tilde W_r(x))
    +\theta x\tilde W_r(x)
    +\alpha\int_0^x\tilde W_r(s)\,ds
    +\alpha K_r, 
    \label{eq: large r tilde integral}
\end{equation}
where $K_r = \frac{1}{\alpha} \left(\frac{\sigma^2}{2}r-F(-p) \right)$. 
Since $\delta<p$, we have $\tilde F(-p)=F(-p)$. 
Observe that since $\tilde F$ is nondecreasing and agrees with $F$ on $(-\infty,-\delta]$, while $F(y)\ge -C(0)$ for $y\le0$, we have that the extended transform admits the lower bound $\tilde F(y)\ge -C(0)$ when $y\in(-p,0)$. 
Therefore, for $x<\tilde\tau_r$, 
\[
\begin{aligned}
    \frac{\sigma^2}{2}\tilde W_r'(x)
    &\geq
    -C(0)
    -p\theta x
    -p\alpha x
    +\frac{\sigma^2}{2}r-F(-p)\\
    &\geq
    -C(0)
    -p(\theta+\alpha)x
    +\frac{\sigma^2}{2}r,
\end{aligned}
\]
where the last inequality follows from $F(-p)\leq0$.

Fix $a>0$. Choose $r_1>0$ sufficiently large such that, for every
$r>r_1$, $2p/r<a$ and
\[
    \frac{2C(0)}{\sigma^2}
    +
    \frac{2p(\theta+\alpha)}{\sigma^2}a
    \leq \frac{r}{2}. 
\]
Then, for every $x<\tilde\tau_r\wedge a$, the inequality becomes $\tilde W_r'(x)\geq\frac{r}{2}>0$. 
Hence $\tilde W_r$ is strictly increasing as long as it remains in
$(-p,0)$. In particular, it cannot exit this interval through the lower
endpoint $-p$.

Moreover, integrating the derivative estimate gives
\[
    \tilde W_r(x)
    \geq
    -p+\frac{r}{2}x,
    \quad 
    x<\tilde\tau_r\wedge a.
\]
If $\tilde\tau_r\leq 2p/r$, then $\tilde\tau_r<\infty$ follows immediately. 
If instead $\tilde\tau_r>2p/r$, then $2p/r<\tilde\tau_r\wedge a$, and hence, we have 
\[
    \tilde W_r\left(\frac{2p}{r}\right)
    \geq
    -p+\frac{r}{2}\frac{2p}{r}
    =0.
\]
However, by the definition of $\tilde \tau_r$, we have $\tilde W_r(x) < 0$ for all $x\in (0, \tilde \tau_r)$, which implies $\tilde W_r(2p/r) < 0$. This is a contradiction. 
% which contradicts the definition of $\tilde\tau_r$. 
Therefore, we have $\tilde\tau_r\leq\frac{2p}{r}<\infty$. Since $\tilde W_r$ cannot exit through $-p$, it must reach zero. Thus, $\tilde W_r(\tilde\tau_r)=0$. 

We now transfer this conclusion to the original solution $W_r$.
Fix $b<\tilde\tau_r$. Since $\tilde W_r$ is continuous and strictly negative on the compact interval $[0,b]$, there exists $\delta_b>0$ such that $\tilde W_r(x)\leq-\delta_b$ for $x\in[0,b]$. 
Choosing the extension parameter sufficiently small so that $0<\delta<\delta_b$, the extended and original transforms coincide along this portion of the trajectory such that $\tilde F(\tilde W_r(x)) = F(\tilde W_r(x))$ for $x\in[0,b]$. By uniqueness of the corresponding initial-value problem, we have $W_r(x)=\tilde W_r(x)$. 
Since $b<\tilde\tau_r$ is arbitrary, it follows that $W_r(x)=\tilde W_r(x)$ for $0\leq x<\tilde\tau_r$. Therefore, $W_r(x)<0$, and
\[
    \lim_{x\to\tilde\tau_r^-}W_r(x)
    =
    \lim_{x\to\tilde\tau_r^-}\tilde W_r(x)
    =0.
\]
By the definition of $c_r$, this implies $c_r=\tilde\tau_r<\infty$ and $\lim_{x\to c_r^-}W_r(x)=0$. This completes the proof.
\end{proof}

Lemmas \ref{comparison lemma}-\ref{large r case} characterize the qualitative behavior of the family $\{W_r\}_{r\geq0}$ as the initial slope $r$ varies. In particular, the comparison result orders the solution profiles with respect to $r$, while Lemmas \ref{small r case} and \ref{large r case} distinguish the behaviors associated with sufficiently small and sufficiently large values of $r$. These results motivate the existence of a critical parameter $r^*$ separating the two regimes. The existence and properties of such a critical parameter are established rigorously in Proposition~\ref{exist r^*} below. 

Figure \ref{W_r with different r values} provides a numerical illustration of these solution profiles. The parameters $r_i$, $i=0,1,\cdots,7$, are arranged in increasing order, with $r_3<r^*<r_4$. The trajectory corresponding to $r^*$ represents the critical profile separating solutions that develop a negative local maximum from those that reach zero in finite time. For this illustration, we take $p=1.1$, $\sigma=1$, $\theta=0.5$, and $\alpha=0.5$, together with the cost function specified in Example~\ref{example: exponential cost}. The differential equation \eqref{wr with integral} is solved numerically using a finite-difference scheme, with the integral term approximated by the trapezoidal rule. 
\begin{figure}[htbp]
    \centering
    \includegraphics[width=0.6\linewidth]{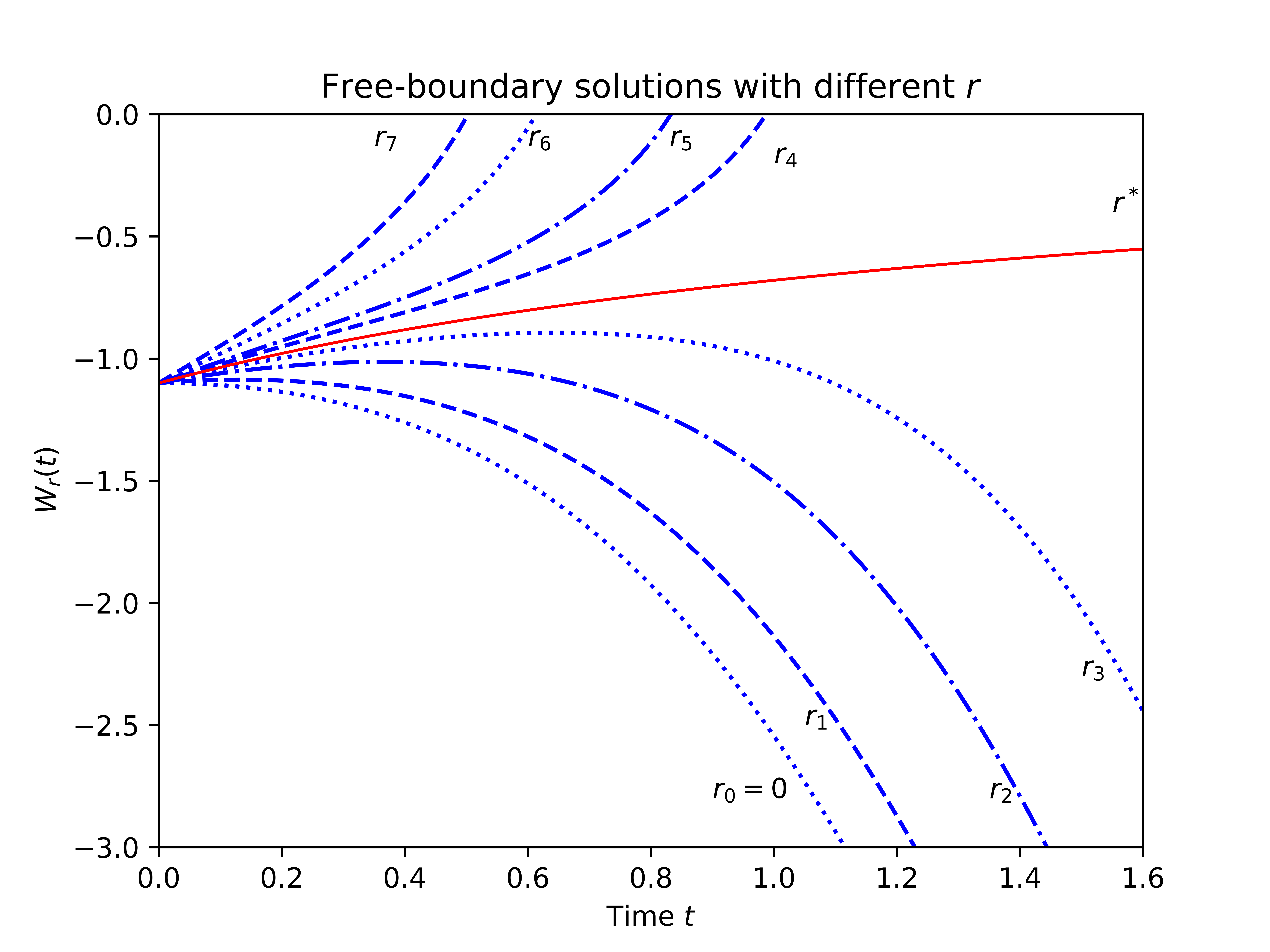}
    \caption{Numerical illustration of the solution profiles $W_r$ for different values of the initial slope $r$.}
    \label{W_r with different r values}
\end{figure}

% Using Lemmas \ref{comparison lemma}-\ref{large r case}, one can obtain an overall solution profile of the second-order differential equation \eqref{wr with integral}. In Figure \ref{W_r with different r values}, we exhibit the sample graphs of $W_r$ with respect to different $r$ values, where $r_i$'s are in ascending order for $i = 0, 1, \cdots, 7$ and $r_3< r^* < r_4$. Moreover, $r^*$ is the one for which we are looking, such that $W_{r^*}$ is the optimal solution to \eqref{wr with integral}. 
% Here we assume $p = 1.1$, diffusion parameter $\sigma = 1$ and drift parameters $\theta = 0.5$. 
% We also assume the cost function presented in Example \ref{example: exponential cost}. 
% % The corresponding Legendre transform has been computed in \eqref{eq: F example}. 
% For the performance measure defined in \eqref{costfunctional}, we take the discounted parameter $\alpha = 0.5$. Moreover, we utilize the finite difference method and the Newton-Cotes formula (Trapezoid rule) as an integral approximation of \eqref{wr with integral}. 
% \begin{figure}[htbp]
%     \centering
%     \includegraphics[width=0.6\linewidth]{Images/W_r(t).png}
%     \caption{Illustrative diagram of solutions $W_r(\cdot)$ with different values of $r$.}
%     \label{W_r with different r values}
% \end{figure}

\begin{proposition}
There exists a value $r^*>0$ such that the corresponding solution $Y_{r^*}$ is defined on $[0,\infty)$, bounded, strictly decreasing, and convex. Moreover, $-p\leq Y_{r^*}'(x)<0$ for $x\geq0$, and $\lim_{x\to\infty}Y_{r^*}'(x)=0$. 
\label{exist r^*}
\end{proposition}

\begin{proof}
Define
\begin{equation}
    r^*
    :=
    \sup\left\{
        r>0:
        W_r \text{ has a negative local maximum before reaching zero}
    \right\}.
    \label{r^*}
\end{equation}
By Lemma \ref{small r case}, the set in \eqref{r^*} is nonempty, whereas Lemma \ref{large r case} implies that it is bounded above. Hence, we have $0<r^*<\infty$. 

First, we show that $W_{r^*}$ has no negative local maximum. Suppose, to the contrary, that there exists $a\in(0,c_{r^*})$ such that $W_{r^*}(a)<0$ and $W_{r^*}'(a)=0$. By the discussion preceding Lemma \ref{touch x axis tangentially}, we have 
\[
    \frac{\sigma^2}{2}W_{r^*}''(a)
    =
    (\theta+\alpha)W_{r^*}(a)<0.
\]
Thus $a$ is a strict local maximum of $W_{r^*}$. Choose $\varepsilon>0$ sufficiently small such that $W_{r^*}'(a-\varepsilon)>0$ and $W_{r^*}'(a+\varepsilon)<0$. 
Note that the continuous dependence of the solution on the initial slope $W_r'(0)=r$ implies that, for all $r>r^*$ sufficiently close to $r^*$, we have $W_r'(a-\varepsilon)>0$ and $W_r'(a+\varepsilon)<0$, and $W_r$ remains negative on
$[a-\varepsilon,a+\varepsilon]$. Consequently, $W_r'$ vanishes at some point in this interval, where $W_r$ has a negative local maximum. This contradicts the definition
of $r^*$ in \eqref{r^*}. Therefore, $W_{r^*}$ has no negative local maximum.

Second, we claim that $W_{r^*}$ cannot have a zero local maximum. 
Since $W_{r^*}'(0)=r^*>0$, and $W_{r^*}$ cannot have a negative local maximum, it follows that
\begin{equation}
    W_{r^*}'(x)>0,
    \quad 
    0\leq x<c_{r^*}.
    \label{eq: critical Wr increasing}
\end{equation}
This is true since if $W_{r^*}'$ vanished for the first time at some $a<c_{r^*}$, then $W_{r^*}(a)<0$ and \eqref{wrequation with tilde F} imply 
\[
    W_{r^*}''(a)
    =
    \frac{2(\theta+\alpha)}{\sigma^2}W_{r^*}(a)<0,
\]
so that $a$ would be a negative local maximum, which has just been excluded. 
If $W_{r^*}(a) = 0$, when by Lemma~\ref{touch x axis tangentially}, this contradicts with the assumption that $W_{r^*}'$ vanishes at $a$. Hence $W_{r^*}$ is strictly increasing while it remains negative.

Third, we intend to exclude the case of $W_{r^*}$ having an endpoint maximum if $c_{r^*} < \infty$. We next show that $c_{r^*}=\infty$. 
Suppose instead that $c_{r^*}<\infty$. By the definition of $c_{r^*}$ and monotonicity of $W_{r^*}$, we have $\lim_{x\to c_{r^*}^-}W_{r^*}(x)=0$. 
Lemma \ref{touch x axis tangentially} then gives $\lim_{x\to c_{r^*}^-}W_{r^*}'(x)>0$. Thus the critical trajectory reaches zero transversally. 
By continuous dependence on the initial slope, this transversal crossing persists under sufficiently small perturbations of $r$. In particular, we can find $\delta_0 >0$ so that whenever $-\delta_0 < r - r^* < 0$, the corresponding trajectory $W_r$ also reaches zero in finite time with positive slope. 
On the other hand, by the definition of $r^*$ together with the comparison structure established in Lemma \ref{comparison lemma}, parameters below the critical value belong to the regime in which the trajectory develops a negative local maximum before reaching zero. This is a contradiction. Therefore, $c_{r^*}=\infty$.

Further, combining this conclusion with \eqref{eq: critical Wr increasing}, we obtain
\begin{equation}
    -p=W_{r^*}(0)
    \leq W_{r^*}(x)<0,
    \quad 
    W_{r^*}'(x)>0,
    \quad  x\geq0.
    \label{eq: critical Wr properties}
\end{equation}
Since $W_{r^*}$ is increasing and bounded above by zero, the limit $\ell:=\lim_{x\to\infty}W_{r^*}(x)$ exists for some $\ell\in[-p,0]$. We claim that $\ell=0$. 

Suppose instead that $\ell<0$. Choose $\varepsilon>0$ such that
$\ell+\varepsilon<0$. Then there exists $x_0>0$ for which
\[
    W_{r^*}(x)\leq\ell+\varepsilon<0,
    \quad  x\geq x_0.
\]
Since $W_{r^*}(x)\in[-p,0)$ and $\int_0^xW_{r^*}(s)\,ds\leq0$, using \eqref{wr with integral}, we therefore obtain, for $x\geq x_0$,
\[
\begin{aligned}
    \frac{\sigma^2}{2}W_{r^*}'(x)
    &=
    F(W_{r^*}(x))
    +\theta xW_{r^*}(x)
    +\alpha\int_0^xW_{r^*}(s)\,ds
    +\alpha K_{r^*}  \\
    &\leq
    \theta x(\ell+\varepsilon)+\alpha K_{r^*}. 
\end{aligned}
\]
Because $\ell+\varepsilon<0$, the right-hand side tends to $-\infty$ as $x\to\infty$. Hence $W_{r^*}'(x)<0$ for all sufficiently large $x$, contradicting \eqref{eq: critical Wr properties}. Therefore, we conclude that
\begin{equation}
    \lim_{x\to\infty}W_{r^*}(x)=0.
    \label{eq: Wr critical limit}
\end{equation}

Last but not least, it remains to establish the boundedness of $Y_{r^*}$. 
Since $W_{r^*}(x)=Y_{r^*}'(x)<0$ and $W_{r^*}'(x)=Y_{r^*}''(x)>0$, the function $Y_{r^*}$ is strictly decreasing and convex on $[0,\infty)$. Moreover, equation \eqref{yrequation} can be written as
\begin{equation}
    \alpha Y_{r^*}(x)
    =
    \frac{\sigma^2}{2}W_{r^*}'(x)
    -F(W_{r^*}(x))
    -\theta xW_{r^*}(x).
    \label{eq: positivity Yrstar}
\end{equation}
Since $W_{r^*}'(x)>0$, $W_{r^*}(x)<0$, and $F(y)\leq0$ for $y\leq0$, every term on the right-hand side of \eqref{eq: positivity Yrstar} is nonnegative. Hence $Y_{r^*}>0$.
Because $Y_{r^*}$ is strictly decreasing,
\[
    0<Y_{r^*}(x)
    \leq
    Y_{r^*}(0)
    =
    K_{r^*}
    =
    \frac{1}{\alpha}
    \left(
        \frac{\sigma^2}{2}r^*-F(-p)
    \right)
    <\infty.
\]
Therefore, $Y_{r^*}$ is bounded on $[0,\infty)$. In addition, we have $-p = Y_{r^*}'(0) \leq Y_{r^*}'(x) <0$ and $\lim_{x\to\infty}Y_{r^*}'(x)=\lim_{x\to\infty}W_{r^*}(x)=0$. 
This completes the proof.
\end{proof}

Now, we are ready to prove Theorem \ref{A solution to the HJB equation}. 

% \textbf{Proof of Theorem \ref{A solution to the HJB equation}.}
\begin{proof}[Proof of Theorem \ref{A solution to the HJB equation}]
For all $x\in[0, \infty)$, let $Q(x) = Y_{r^*}(x)$, where $r^*>0$ is the critical parameter obtained in Proposition~\ref{exist r^*}. 
By construction, $Y_{r^*}$ solves \eqref{yrequation}. Hence, \eqref{formal HJB}. Moreover, Proposition \ref{exist r^*} shows that $Q$ is bounded, strictly decreasing, and convex on $[0,\infty)$. In particular, $-p\leq Q'(x)<0$ for $x\geq0$, and $\lim_{x\to\infty}Q'(x)=0$. Therefore, $Q$ satisfies all the assertions of Theorem \ref{A solution to the HJB equation}. This completes the proof. 
\end{proof}

\section{Proofs from Section~\ref{Subsec2 of Sec3 in CH4: Optimal Solution of the DCP}}
\subsection{Proof of Lemma~\ref{verification lemma}}
\label{appendix:verificaiton}

\begin{proof}[Proof of Lemma~\ref{verification lemma}]
% \textbf{Proof.}
Define a function $f(t, x) = e^{-\alpha t} Q(x)$, where $Q$ satisfies the assumptions mentioned above and $\alpha>0$ is a fixed constant. 
Employing the It\^{o}'s formula to obtain
\begin{equation*}
\begin{aligned}
e^{-\alpha t}Q(X(t)) = Q(x) &+ \int_0^{t} e^{-\alpha s} \bigg[\frac{\sigma^2}{2}Q''(X(s)) - u(s)Q'(X(s))\\
&\quad\quad\quad\quad\quad- \theta X(s)Q'(X(s))-\alpha Q(X(s)) \bigg]ds\\
&+ \sigma \int_0^{t} e^{-\alpha s}Q'(X(s))dW(s) + \int_0^t e^{-\alpha s} Q'(X(s))dL(s). 
\end{aligned}
\end{equation*}
Observe that the sufficient smoothness of $Q(\cdot)$ and the boundedness of $Q'(\cdot)$ imply that the stochastic integral term has an expected value of zero, since one can obtain 
\begin{equation*}
\begin{aligned}
    \mathbb E\left[ \int_0^t (\sigma e^{-\alpha s}Q'(X(s)))^2ds \right] &= \sigma^2 \mathbb E\left[\int_0^t e^{-2\alpha s}({Q'}(X(s)))^2ds \right] \\
    &\leq {M}^2 \sigma^2 \int_0^t e^{-2\alpha s} ds < \infty,  
\end{aligned}
\end{equation*}
which guarantees that the expected value of the stochastic integral is equal to zero. 

Now taking the expected value on both sides to obtain the following equality:
\begin{equation*}
\begin{aligned}
    \mathbb E\left[e^{-\alpha t}Q(X(t))\right] = Q(x) 
    &+ \mathbb E\bigg[\int_0^{t} e^{-\alpha s} \bigg(\frac{\sigma^2}{2}Q''(X(s)) - u(X(s))Q'(X(s))\\
    &\quad\quad\quad\quad\quad\quad\quad  - \theta X(s)Q'(X(s))-\alpha Q(X(s)) \bigg)ds \bigg]\\
    &+ \mathbb E\left[\int_0^t e^{-\alpha s} Q'(X(s))dL(s) \right]. 
\end{aligned}
\end{equation*}
Since $Q$ satisfies the equation \eqref{rewritten HJB}, we have that for all $u \geq 0$, 
\begin{equation*}
    \frac{\sigma^2}{2} Q''(x) - \theta x Q'(x)  - \alpha Q(x) = F(Q'(x)) \geq uQ'(x) - C(u), 
\end{equation*}
which further implies the following inequality:
\begin{equation*}
    \frac{\sigma^2}{2} Q''(x) -uQ'(x)- \theta x Q'(x)  - \alpha Q(x) \geq  - C(u). 
\end{equation*}
Together with $Q'(0) = -p$ and properties of the regulator $L(\cdot)$, we have
\begin{equation*}
    \mathbb E\left[e^{-\alpha t}Q(X(t))\right] \geq 
    Q(x) - \mathbb E\left[\int_0^{t} e^{-\alpha s} C(u(s))ds +p\int_0^t e^{-\alpha s} dL(s) \right]. 
\end{equation*}
Rearranging the above inequality to obtain
\begin{equation*}
    \mathbb E\left[\int_0^{t} e^{-\alpha s} C(u(s))ds + p\int_0^t e^{-\alpha s} dL(s) \right] \geq 
    Q(x) - \mathbb E\left[e^{-\alpha t}Q(X(t))\right] . 
\end{equation*}
Let $t$ approach positive infinity. Since $Q$ is bounded, $\lim_{t\to\infty} \mathbb E[e^{-\alpha t}Q(X(t))] = 0$. 
Moreover, we observe that the left-hand side coincides with the infinite-horizon discounted cost functional $J(x, u)$ defined in \eqref{costfunctional}. Therefore, we conclude that $J(x, u) \geq Q(x)$ leads to $V(x) \geq Q(x)$ by taking the infimum over all non-negative controls $u$. 
\end{proof}

\subsection{Proof of Theorem~\ref{An optimal control}}

\begin{proof}[Proof for Theorem~\ref{An optimal control}]
The feedback control $u^*$ is well defined since $Q'(x)<0$ for all $x\geq0$. Under the assumptions imposed on $C$, Assumption~\ref{assumption: prelimit control costs}, the corresponding reflected equation \eqref{optimal SDE} admits a weak solution; see Section 4.9 of \cite{xie2022topics} for the corresponding existence argument, where an interesting proof using a comparison result is provided.

Apply the It\^{o}'s formula to $e^{-\alpha t}Q(X^*(t))$, we obtain
\begin{equation*}
\begin{aligned}
    e^{-\alpha t}Q(X^*(t)) = Q(x) &+ \int_0^t e^{-\alpha s} \bigg[\frac{\sigma^2}{2}Q''(X^*_s) - \alpha Q(X^*_s) \\
    &\quad\quad\quad\quad\quad- u^*(s)Q'(X^*_s) - \theta X^*_s Q'(X^*_s)\bigg]ds \\
    &+\sigma \int_0^{t} e^{-\alpha s}Q'(X^*_s)dW(s) + \int_0^t e^{-\alpha s} Q'(X^*_s)dL_s^*,  
\end{aligned}
\end{equation*}
where $\alpha>0$ is a fixed constant. 
By \eqref{rewritten HJB with F be substituted}, the drift integrand equals $-C(u^*(X^*(s))$. 
Furthermore, since the stochastic integral term has an expected value of zero, by taking the expected value on both sides and some algebraic manipulations, we further obtain
\begin{equation*}
% \begin{aligned}
    Q(x) - \mathbb E\left[e^{-\alpha t}Q(X^*(t))\right]  
    % &=  E\left[\int_0^t e^{-\alpha s}C(u^*(s))ds + p\int_0^t e^{-\alpha s} dL_s^*\right]\\
     = \mathbb E\left[\int_0^t e^{-\alpha s} \left(C(u^*(s))ds + pdL_s^*\right)\right].
% \end{aligned}
\end{equation*}
The boundedness of $Q(\cdot)$ guarantees that $\lim_{t\to\infty} \mathbb E[e^{-\alpha t}Q(X^*(t))]=0$. Next, let $t$ tend to positive infinity. We complete the proof. 
\end{proof}

\section{Proofs from Section~\ref{Sec3 in CH4: Asymptotic Optimality}}
\label{appendix: asymptotic optimality}

\subsection{Proof of Theorem~\ref{thm: asymptotic lower bound}}

\begin{proof}[Proof of Theorem \ref{thm: asymptotic lower bound}]
For simplicity, we assume $x_n\equiv x$. 
Proposition \ref{diffusion limit} implies that $(\hat{V}_n, \hat{L}_n)$ converges weakly to $(X, L)$ in $D^2[0, \infty)$ and with the help of the Skorokhod's representation theorem (see Theorem 3.2.2 in \cite{whitt2006stochastic}), we can simply assume that $\lim_{n\to\infty}( \hat{V}_n(t), \hat{L}_n(t))  = (X(t), L(t))$ for all $t\geq 0$. Moreover, we obtain that $(X, L)$ is the unique strong solution to the stochastic differential equation \eqref{limiting process} by Proposition \ref{diffusion limit}. Thus, $(x, u, X)$ is an admissible control, i.e., $(x, u, X)\in\mathcal{A}(x)$. 

Consider the cost functional of the QCP, \eqref{cost functional of QCP}, by integration by parts, we reduce the last term as follows: 
\begin{equation*}
    p \mathbb E\left[\int_0^{\infty} e^{-\alpha t} d\hat{L}_n(t)\right] 
    = p\alpha  \mathbb E\left[\int_0^{\infty} e^{-\alpha t} \hat{L}_n(t)dt\right]. 
\end{equation*}
Substituting the last term, the cost functional of the QCP \eqref{cost functional of QCP} becomes \eqref{eq: QCP cost asymptotic section}. 
Similarly, the cost functional \eqref{costfunctional} could be written as \eqref{eq: rewritten cost functional of limiting process} as discussed above. 

First, we consider the idle-time cost. 
Since $(\hat{V}_n, \hat{L}_n)$ converges to $(X, L)$ a.s. as $n$ tends to positive infinity and non-negativity of $e^{-\alpha t} \hat{L}_n(t)$, the Fatou's lemma renders
\begin{equation}
    \mathbb E\int_0^{\infty} e^{-\alpha t} L(t) dt \leq \liminf_{n\to \infty} \mathbb E\int_0^{\infty} e^{-\alpha t} \hat{L}_n(t)dt. 
    \label{1st term by Fatou}
\end{equation}

Next, we consider the running-cost term. 
By Assumption~\ref{assumption: asymptotic control un}, for every $K>0$, we have $\sup_{y\in [0, K]} \lvert u_n(y) - u(y)\rvert \to 0$ as $n\to\infty$. 
By applying Lemma 4.6 in \cite{lee2011convergence}, Proposition \ref{diffusion limit}, and continuous mapping theorem, for each $\epsilon>0$, we obtain the following inequalities:
\begin{equation*}
\begin{aligned}
&\mathbb P\left(\|u_n(\hat{V}_n)-u(X)\|_T>\epsilon\right) \\
&\leq \mathbb P\left(\|u_n(\hat{V}_n)-u(\hat{V}_n)\|_T+\|u(\hat{V}_n)-u(X)\|_T>\epsilon\right)\\
&\leq \mathbb P\left(\|u_n(\hat{V}_n)-u(\hat{V}_n)\|_T>\frac{\epsilon}{2}\right) + \mathbb P\left(\|u(\hat{V}_n)-u(X)\|_T>\frac{\epsilon}{2}\right), 
\end{aligned}
\end{equation*}
which implies $\|u_n(\hat{V}_n)-u(X)\|_T\to 0$ in probability as $n\to\infty$ for any $T>0$. 
Recall from Assumption~\ref{assumption: prelimit control costs} that $C_n \to \bar C$ locally uniformly on $\mathbb R$, and $C = \bar C\rvert _{[0, \infty)}$. 
Using the continuous mapping theorem again, we obtain that $C_n(u_n(\hat{V}_n(t)))$ converges to $C(u(X(t)))$ in probability as $n$ tends to positive infinity, and they are both non-negative. 

Therefore, by Fatou's lemma, we have
\begin{equation}
    \mathbb E\int_0^{\infty} e^{-\alpha t} C(u(X(t)))dt \leq \liminf_{n\to\infty} \mathbb E\int_0^{\infty} e^{-\alpha t} C_n(u_n(\hat{V}_n(t)))dt. 
\end{equation}
Hence, combining with \eqref{1st term by Fatou}, and using the elementary inequality
\[
    \liminf_{n\to\infty}(a_n+b_n)
    \geq
    \liminf_{n\to\infty}a_n
    +
    \liminf_{n\to\infty}b_n
\]
for nonnegative sequences, we conclude that 
\begin{equation}
    \liminf_{n\to \infty} J_n(x_n, u_n)  = \liminf_{n\to \infty} J(x, u_n) \geq J(x, u). 
\end{equation}
Since $(x, u, X)\in\mathcal{A}(x)$ and $J(x, u)\geq V(x)$ by the definition of the value function in \eqref{value function}, we conclude \eqref{eq: asymptotic lower bound first}. 
This completes the proof. 
\end{proof}

\subsection{Proof of Theorem~\ref{Asymptotic optimality}}

To deduce the asymptotic optimality, we intend to simplify the scaled state process \eqref{queueing systems}. 
We denote
\begin{align}
    \xi_n(t) &:= \hat{A}_n(t) + \hat{M}_n^v(\bar{A}_n(t)) - \hat{M}_n^d(\bar{A}_n(t)), \label{Eq: xi_n} \ \text{and} \\
    \epsilon_n(t) &:= \frac{1}{\sqrt{n}}\int_0^t F_n\left(\frac{\hat{V}_n(s-)}{\sqrt{n}}\right)dA_n(s) - \theta \int_0^t  \hat{V}_n(s)ds, 
    % + \int_0^t\left(u_n(\hat{V}_n(s))-u(\hat{V}_n(s))\right)ds, 
    \label{Eq: epsion_n(t)}
\end{align}
for all $t\geq0$.
By Lemma 4.5 and Lemma 4.7 in \cite{lee2011convergence}, we can simply assume that $\xi_n(\cdot)$ converges weakly to $\sigma W(\cdot)$ in $ D[0, \infty)$ as $n\to\infty$ and $\|\epsilon_n\|_T\to 0$ in probability as $n\to\infty$. Hence, \eqref{queueing systems} can be rewritten as
\begin{equation}
    \hat{V}_n(t) = x_n + \xi_n(t) - \epsilon_n(t) - \int_0^t u_n(\hat{V}_n(s))ds - \theta \int_0^t \hat{V}_n(s)ds + \hat{L}_n(t),
    \label{simplified queue}
\end{equation}
where $u_n(x)=\sqrt{n}\left(1-\lambda_n\left(\frac{x}{\sqrt{n}}\right)\right)$ for all $x\geq0$, $\xi_n(\cdot)$ and $\epsilon_n(\cdot)$ are defined in \eqref{Eq: xi_n} and \eqref{Eq: epsion_n(t)}, and they also satisfy above convergent assumptions.

We next establish that the $\hat{L}_n$ process has a finite second moment for all $n \geq 0$. This result facilitates a direct proof of the uniform integrability of the relevant integrands in Theorem \ref{Asymptotic optimality}. Without loss of generality and for simplicity, we assume a deterministic initial state $x_n \equiv x$ throughout this section.

\begin{lemma}
Consider the $n$th system $(\hat{V}_n, \hat{L}_n)$. We also assume $x_n$ converges to $x$ as $n\to\infty$. Then, we conclude that 
\begin{equation}
    \mathbb E\left[\hat{L}_n(T)^2\right] \leq C_0 +  C(1+T^{2(m+1)}),
    \label{second moment of L}
\end{equation}
for all $T>0$, where $C_0>0$ and $C > 0$ are constants independent of $n$ and $m>2$. 
\label{second moment bound of L}
\end{lemma}

\begin{proof}
% \textbf{Proof.}
Let the Skorokhod map $\Gamma_0: D[0, \infty)\mapsto D[0, \infty)$ (see \cite{kruk2007explicit}):
\begin{equation}
    \Gamma_0(\phi)(t) = \phi(t) + \sup_{s\in[0, t]} [-\phi(s)]^+, 
    \label{Skorokhod map}
\end{equation}
where $\phi(t)\in  D[0, \infty)$ for all $t\geq0$. We further introduce a process $\hat{Z}_n(t) = \hat{V}_n(t)-\hat{L}_n(t)$ for all $t\geq 0$, where $\hat{V}_n(\cdot)$ and $\hat{L}_n$ are defined in \eqref{queueing systems}. Thus we can obtain the Skorokhod decomposition $(\hat{V}_n, \hat{L}_n)$ of $\hat{Z}_n$ (see \cite{kruk2007explicit} or Section 4.1 in \cite{lee2011convergence}) and $\Gamma_0(\hat{Z}_n)(t) = \hat{V}_n(t)$ for all $t\geq0$. Since the Skorokhod map has Lipschitz property (see \cite{kruk2007explicit}), i.e., $\sup_{t\in[0, T]}\lvert\Gamma_0(\phi)(t)\rvert\leq 2\|\phi\|_T$, we can obtain the following inequality:
\begin{equation*}
\begin{aligned}
    \mathbb E[\hat{L}_n(T)^2] &= \mathbb E\left[\|\hat{L}_n\|_T^2\right] 
    = \mathbb E\left[\|\Gamma_0(\hat{Z}_n)-\hat{Z}_n\|_T^2\right] \\
    &\leq 2 \mathbb E\left[\|\Gamma_0(\hat{Z}_n)\|_T^2+\|\hat{Z}_n\|_T^2\right]\\
    &\leq 8 \mathbb E\left[\|\hat{Z}_n\|_T^2\right] + 2 \mathbb E\left[\|\hat{Z}_n\|_T^2\right] 
    = 10 \mathbb E\left[\|\hat{Z}_n\|_T^2\right]. 
\end{aligned}
\end{equation*}
Observe that the Skorokhod map provides a straightforward approach in proving the second moment bound. Therefore, it suffices to show the second moment bound of $\hat{Z}(\cdot)$ is finite. Since $\hat{Z}_n(t) = \hat{V}_n(t)-\hat{L}_n(t)$ for all $t\geq0$ and by \eqref{simplified queue}, we have 
\begin{equation}
    \hat{Z}_n(t) = x_n + \xi_n(t) - \epsilon_n(t) - \int_0^t u_n(\Gamma_0(\hat{Z}_n)(s))ds
    - \int_0^t \theta \Gamma_0(\hat{Z}_n)(s)ds, 
    \label{Z_n}
\end{equation}
where $\xi_n(\cdot)$, $\epsilon_n(\cdot)$, and $u_n(\cdot)$ are defined in \eqref{simplified queue}. Using the expressions of $\epsilon_n(\cdot)$ and $u_n(\cdot)$ defined above and assumption $x_n\equiv x$, \eqref{Z_n} can be rewritten as the following: 
\begin{equation*}
\begin{aligned}
    \hat{Z}_n(t) &= x + \xi_n(t) -\frac{1}{\sqrt{n}} \int_0^t F_n \left(\frac{\Gamma_0(\hat{Z}_n)(s-)}{\sqrt{n}}\right) dA_n(s) 
    \\
    &\quad 
    +\sqrt{n} \int_0^t \left[\lambda_n\left(\frac{\Gamma_0(\hat{Z}_n)(s)}{\sqrt{n}}\right)-1\right]ds, 
\end{aligned}
\end{equation*}
By Lemma 6.4 and Proposition 6.10 in \cite{lee2011convergence}, we can easily obtain the second moment bound of $\hat{Z}_n(\cdot)$:
\begin{equation}
    \mathbb E\left[\|\hat{Z}_n(t)\|_T^2\right] \leq 8x^2 + M(1+T^{2(m+1)}), 
    \label{second moment of hat Zn}
\end{equation}
where $x\geq0$ and $M>0$ are some constants independent of $n$, $T$, and $m>2$. Consequently, we have
\begin{equation*}
    \mathbb E\left[\hat{L}_n(T)^2\right] \leq C_0+ C(1+T^{2(m+1)}),
\end{equation*}
for $T>0$, where $C_0=80 x^2$ and $C=10 M > 0$ are constants independent of $n$ and $m>2$. 
\end{proof}

With the above preliminaries established, we can demonstrate the asymptotic lower bound and attainability result, where we may verify the uniform integrability of relevant integrands under a specific control $u^*$. 
% For full proofs of the main theorems and additional details, we refer the reader to the extended version of the paper \cite{xie2022single}. 

\begin{proof}[Proof of Theorem \ref{Asymptotic optimality}]
Without loss of generality we assume $x_n\equiv x$. 
Proposition \ref{diffusion limit} implies that $(\hat{V}^*_n, \hat{L}^*_n)$ converges weakly to $(X^*, L^*)$ in $D^2[0, \infty)$ and with the help of the Skorokhod's representation theorem (see Theorem 3.2.2 in \cite{whitt2006stochastic}), we can simply assume that $\lim_{n\to\infty}( \hat{V}^*_n(t), \hat{L}^*_n(t))  = (X^*(t), L^*(t))$ for all $t\geq 0$. Moreover, we obtain that $(X^*, L^*)$ is the unique strong solution to the stochastic differential equation \eqref{limiting process} by Proposition~\ref{diffusion limit}.

Since $X^*$ and $L^*$ have continuous sample paths, the Skorokhod representation theorem allows us to work on a common probability space on which, for every $T>0$,
\begin{equation}
    \|\hat V_n^*-X^*\|_T
    +
    \|\hat L_n^*-L^*\|_T
    \to 0
    \quad \text{a.s.},
    \label{eq: optimal processes a.s. convergence}
\end{equation}
which implies
\begin{equation}
    \|u_n^*(\hat V_n^*)-u^*(X^*)\|_T
    \to 0
    \quad \text{a.s.}
    \label{eq: optimal composed controls convergence}
\end{equation}
for every $T>0$.

We first consider the running-cost component. Recall from Assumption~\ref{assumption: prelimit control costs} that $C_n\to\bar C$ locally uniformly on $\mathbb R$, and that
$C=\bar C\mid_{[0,\infty)}$. Since $u^*\geq0$, \eqref{eq: optimal composed controls convergence} gives
\begin{equation}
    C_n\bigl(u_n^*(\hat V_n^*(t))\bigr)
    \to 
    C\bigl(u^*(X^*(t))\bigr)
    \quad \text{a.s.}
    \label{eq: optimal running cost pointwise}
\end{equation}
for every $t\geq0$. 
Hence, to streamline the argument and for simplicity, we henceforth consider the special case $C_n \equiv \bar C$ with the running cost of the limiting DCP $C=\bar C\mid_{[0,\infty)}$. The more general case of locally uniformly convergent prelimit costs can be treated with the same argument, provided the corresponding uniform-integrability conditions are imposed.

Observe that when $u_n^*(\cdot) \geq 0$, we observe that $\bar C = C$ and the integrand of the first term in \eqref{cost functional of QCP} is dominated by $e^{-\alpha t}C(0)$ as the following:
\begin{equation*}
    \lvert e^{-\alpha t} C(u_n^*(\hat{V}_n^*(t)))\rvert \leq e^{-\alpha t} C(0), 
\end{equation*}
where $C(0)$ is the maximum of the cost function $C(x)$ for all $x\geq0$ as we assumed. We intend to verify the uniform integrability of the dominated function. It is trivial to see that
\begin{equation*}
    \int_0^{\infty} e^{-\alpha t} C(0) dt \leq \frac{C(0)}{\alpha}<\infty. 
\end{equation*}

Therefore, by applying the generalized dominated convergence theorem, we obtain 
\begin{equation}\label{eq: convergence of C part}
    \lim_{n\to\infty} \mathbb E\left[\int_0^{\infty} e^{-\alpha t}C_n(u_n^*(\hat{V}_n^*(t)))dt\right] 
    = \mathbb E\left[\int_0^{\infty} e^{-\alpha t}C(u^*(X^*(t)))dt\right]. 
\end{equation}

We next consider the idle-time component in \eqref{eq: QCP cost asymptotic section}. Analogously, identical to the proof of Theorem \ref{thm: asymptotic lower bound}, integration by parts suggests that
\begin{equation*}
    \int_0^{\infty} e^{-\alpha t}d\hat{L}_n^*(t) = \alpha\int_0^{\infty} e^{-\alpha t} \hat{L}_n^*(t)dt. 
\end{equation*}
Since Lemma \ref{second moment bound of L} implies that
\begin{equation*}
    \mathbb E\left[\hat{L}_n^*(T)^2\right] \leq C_0+ C(1+T^{2(m+1)}),
\end{equation*}
for some $T>0$, where $C_0>0$ and $C > 0$ are constants independent of $n$ and $m>l$. Therefore, the Fatou's lemma suggests
\begin{equation}
    \mathbb E\left[L^*(T)^2\right] \leq \liminf_{n\to\infty} E\left[\hat{L}_n^*(T)^2\right] \leq C_0+ C(1+T^{2(m+1)}), 
\end{equation}
where $C_0>0$, $C > 0$, and $m>2$ are defined in \eqref{second moment of L}. Moreover, we obtain the following inequality: 
\begin{equation*}
    E\left[\int_0^{\infty}e^{-\alpha t} \hat{L}_n^*(t)^2dt\right] \leq C_0\int_0^{\infty} e^{-\alpha t} dt + 
    C \int_0^{\infty} e^{-\alpha t} (1+t^{2(m+1)})dt < \infty, 
\end{equation*}
which yields the uniform integrability of $(\hat{L}_n^*)$. Therefore, let $n\to\infty$, we obtain
\begin{equation*}
% \begin{aligned}
    \lim_{n\to\infty} E\left[\int_0^{\infty} e^{-\alpha t} d\hat{L}_n^*(t) \right] = \alpha \lim_{n\to\infty} E\left[\int_0^{\infty} e^{-\alpha t} \hat{L}_n^*(t)dt \right] 
    % \\
    % &=\alpha E\left[\int_0^{\infty} e^{-\alpha t} L^*(t)dt \right]
    % \\
    =E\left[\int_0^{\infty} e^{-\alpha t} dL^*(t) \right]. 
% \end{aligned}
\end{equation*}
Together with \eqref{eq: convergence of C part}, we conclude that
\begin{equation*}
    \lim_{n\to\infty} J(x_n, u_n^*) = \lim_{n\to\infty} J(x, u_n^*) = J(x, u^*). 
\end{equation*}
Moreover, the DCP suggests $J(x, u^*, X^*, L^*)=V(x)$ by Theorem \ref{An optimal control}. 
This completes the proof. 

\end{proof}

\section{Additional Numerical Experiments}
\label{appendix:additional-numerical}

\subsection{Sensitivity Analysis}
\label{sec:sensitivity}

We next examine how the optimal state-feedback intensity policy responds to changes in the model parameters. We vary $p$ and $\theta$ individually while holding the remaining parameters at their baseline values in Table~\ref{tab:baseline_parameters}. 
For every parameter configuration, the critical value $r^*$ is recomputed using Algorithm~\ref{alg:ivp_hjb}. Consequently, each reported feedback policy is obtained from the HJB equation corresponding to that particular parameter configuration rather than from the baseline value of $r^*$. 

In Figure~\ref{fig:sensitivity_p}, the idleness penalty $p$ directly enters the boundary condition $Q'(0)=-p$. It also determines the optimal control at the origin through $u^*(0)=F'(-p)$. For $0<p\leq 2$, \eqref{eq: exponential conjugate derivative} and \eqref{eq: optimal feedback} give
\[
u^*(0)
=
-\log\left(\frac{p}{2}\right)
\geq0.
\]
Thus, increasing the idleness penalty decreases the intensity-control intensity near the reflecting boundary. This behavior is consistent with the interpretation of $p$: when idleness becomes more costly, the controller has a stronger incentive to maintain incoming workload and therefore restricts arrivals less aggressively.

In Figure~\ref{fig:sensitivity_theta}, the reneging parameter $\theta$ enters the controlled diffusion through the state-dependent drift term $-\theta X(t)$ in \eqref{control system}. Increasing $\theta$ strengthens the endogenous downward drift of the offered waiting time. As a result, changes in $\theta$ alter the amount of intensity control required at a given state. Figure~\ref{fig:sensitivity_theta} illustrates the corresponding changes in the optimal feedback profile.

% \begin{figure}[htbp]
% \centering
% \includegraphics[width=0.65\textwidth]{Images/sensitivity_p.png}
% \caption{Sensitivity of the optimal state-feedback control $u^*$ to the idleness penalty $p$.}
% \label{fig:sensitivity_p}
% \end{figure}

\begin{figure}[htbp]
\centering

\begin{subfigure}[t]{0.48\textwidth}
    \centering
    \includegraphics[width=\textwidth]{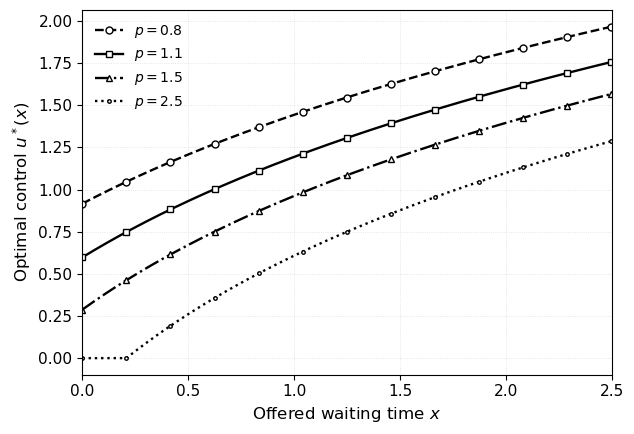}
    \caption{Sensitivity to the idleness penalty $p$.}
    \label{fig:sensitivity_p}
\end{subfigure}
\hfill
\begin{subfigure}[t]{0.48\textwidth}
    \centering
    \includegraphics[width=\textwidth]{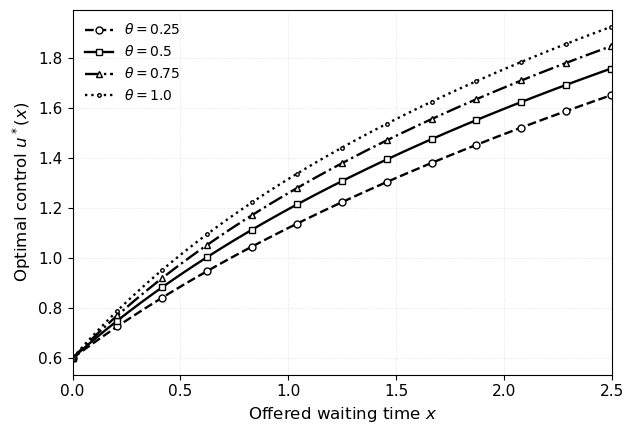}
    \caption{Sensitivity to the reneging parameter $\theta$.}
    \label{fig:sensitivity_theta}
\end{subfigure}

\caption{Sensitivity of the optimal state-feedback control $u^*$ to changes in the model parameters.}
\label{fig:sensitivity_analysis}

\end{figure}

% \begin{figure}[htbp]
% \centering
% \includegraphics[width=0.65\textwidth]{Images/sensitivity_theta.png}
% \caption{Sensitivity of the optimal state-feedback control $u^*$ to the reneging parameter $\theta$.}
% \label{fig:sensitivity_theta}
% \end{figure}

% Similarly, changes in $\sigma$ modify the magnitude of stochastic fluctuations in the offered waiting time, whereas changes in $\alpha$ modify the relative importance assigned to future costs. In each case, the critical parameter $r^*$ and the resulting HJB solution are recomputed before evaluating the corresponding optimal feedback policy. The sensitivity experiments therefore illustrate how both the stochastic dynamics and the cost criterion affect the shape of the optimal state-dependent intensity policy. 

\subsection{Numerical Comparison of the DCP-Optimal Policy}
\label{sec:policy_comparison}

We next compare the optimal state-feedback policy $u^*$ with selected alternative intensity-control policies. The purpose of these experiments is twofold. First, we illustrate the trade-off between the running cost and the idleness cost under different control policies. Second, we examine the performance of admissible feedback policies obtained by perturbing $u^*$. These comparisons provide a numerical assessment of the DCP-optimal policy characterized by the HJB equation.

For an admissible policy $u$, recall that the discounted cost can be decomposed as
\begin{equation}
J(x,u)
=
J_{\mathrm{r}}(x,u)
+
J_{\mathrm{i}}(x,u),
\label{eq:cost_decomposition}
\end{equation}
where
\begin{align}
J_{\mathrm{r}}(x,u)
&=
\mathbb{E}
\left[
\int_0^\infty
e^{-\alpha t} C(u(X(t)))\,dt
\right],
\label{eq:control_cost_component}\\
J_{\mathrm{i}}(x,u)
&=
p\,
\mathbb{E}
\left[
\int_0^\infty
e^{-\alpha t}\,dL(t)
\right].
\label{eq:idleness_cost_component}
\end{align}
Here, $J_{\mathrm{r}}$ represents the expected discounted running cost, whereas $J_{\mathrm{i}}$ represents the expected discounted idleness cost generated by the regulating process $L$ at the origin.

We approximate these quantities by Monte Carlo simulation of the regulated diffusion \eqref{control system}, written in differential form as
\begin{equation}
dX(t)
=
-\bigl(u(X(t))+\theta X(t)\bigr)\,dt
+\sigma\,dW(t)
+dL(t).
\label{eq:simulated_diffusion}
\end{equation}
The state process is simulated 
% using an Euler-Maruyama discretization followed 
by a discrete reflection step at the origin. 
If $X_k$ denotes the simulated state at time $t_k=k\Delta t$, we first compute the unconstrained update
\begin{equation*}
\tilde X_{k+1}
=
X_k
-
\bigl(u(X_k)+\theta X_k\bigr)\Delta t
+
\sigma\sqrt{\Delta t}\,Z_k,
\end{equation*}
where $\{Z_k\}$ are independent standard normal random variables. The corresponding regulating increment and reflected state are approximated by
\begin{equation*}
\Delta L_k
=
\max\{-\tilde X_{k+1},0\},
\quad 
X_{k+1}
=
\max\{\tilde X_{k+1},0\}.
\end{equation*}
Thus, any proposed Euler step below the origin is reflected back to zero, and the discrepancy is recorded as an increment of the regulating process. 
Note that common random numbers are used across competing policies so that all policies are evaluated using the same Brownian increments.

\begin{example}[Comparison with constant controls]
We first compare $u^*$ with the constant feedback policies
\[
u(x)\equiv\bar u,
\quad 
\bar u\in\{0,0.5,1.5,3.0\}.
\]
For each policy, we estimate the total expected discounted cost together with the running-cost and idleness-cost components in \eqref{eq:cost_decomposition}. Figure~\ref{fig:cost_comparison} reports the resulting decomposition. The two components are stacked so that the total height of each bar represents the estimated total discounted cost. 
\end{example}

\begin{figure}[htbp]
\centering
\includegraphics[width=0.65\textwidth]{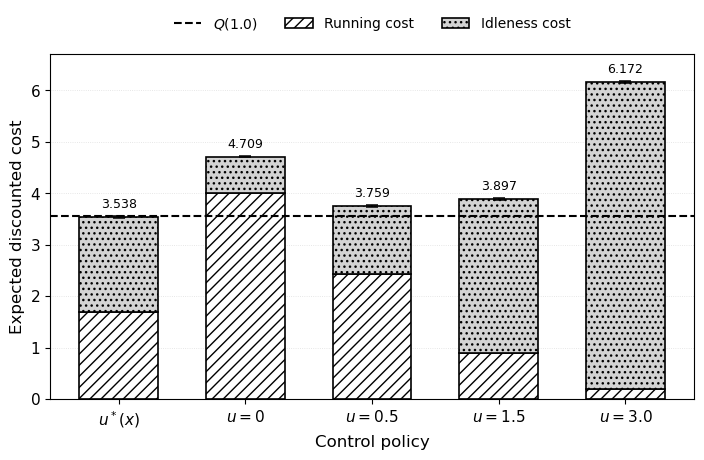}
\caption{Comparison of the expected discounted running and idleness costs under the DCP-optimal feedback policy $u^*$ and selected constant intensity-control policies.}
\label{fig:cost_comparison}
\end{figure}

The numerical comparison illustrates the competing effects of arrival-intensity control. 
A more aggressive intensity-control policy decreases the arrival intensity and can reduce the offered waiting time, but it changes the running cost and may also increase the frequency with which the regulated process reaches the origin. Conversely, a less aggressive policy may reduce the idleness cost while producing a larger running cost. The numerical results show that the state-dependent policy $u^*$ achieves a lower estimated total discounted cost than the selected constant policies. In particular, the comparison illustrates how the feedback structure of $u^*$ balances the running and idleness costs rather than minimizing either component separately. 

We next perform a more stringent comparison using admissible state-dependent perturbations of the DCP-optimal feedback. The simulation parameters are summarized in Table~\ref{tab:simulation_parameters}. For each policy, we simulate $N=8{,}000$ trajectories of the regulated diffusion over $[0,T]$ using the time step $\Delta t=0.001$. The same Brownian increments are used across all perturbations.

\begin{table}[htbp]
\centering
\caption{Parameters for the perturbation experiment.}
\label{tab:simulation_parameters}
\begin{tabular}{ccc}
\hline
Parameter & Value & Description \\
\hline
$x_0$       & $1.0$     & Initial offered waiting time \\
$T$         & $20$      & Simulation horizon \\
$\Delta t$  & $0.001$   & Time step \\
$N$         & $8{,}000$ & Number of sample paths \\
$\delta$
&
\makecell{$-0.40,-0.25,-0.15,-0.10,-0.05,0,$\\
$0.05,0.10,0.15,0.25,0.40$}
& Perturbation parameters\\
\hline
\end{tabular}
\end{table}

\begin{example}[Admissible perturbations of the optimal policy]
For each perturbation parameter $\delta$, define
\begin{equation}
u_\delta(x)
=
\bigl[u^*(x)+\delta\bigr]^+
=
\max\{u^*(x)+\delta,0\}.
\label{eq:perturbed_policy}
\end{equation}
The positive-part operator ensures that every perturbed feedback remains in the admissible control region. The case $\delta=0$ corresponds to the HJB-optimal policy $u^*$, while positive perturbations increase the control intensity. Negative perturbations decrease the control intensity, subject to the lower bound $u_\delta(x)\geq0$.

For each $\delta$, we compute the discretized cost functional
\begin{equation}
\hat J_N(x,u_\delta)
=
\frac{1}{N}
\sum_{j=1}^{N}
\left[
\int_0^T
e^{-\alpha t}
C\bigl(u_\delta(X_j(t))\bigr)\,dt
+
p\int_0^T
e^{-\alpha t}\,dL_j(t)
\right].
\label{eq:MC_perturbed_cost}
\end{equation}
The horizon $T$ is chosen sufficiently large relative to the discount factor so that the contribution of the remaining discounted tail is small for the purpose of the numerical comparison.
\end{example}

Figure~\ref{fig:policy_perturbation} plots the estimated discounted cost as a function of $\delta$. The point $\delta=0$, highlighted in red, corresponds to the DCP-optimal feedback policy. The error bars represent $95\%$ confidence intervals.

\begin{figure}[htbp]
\centering
\includegraphics[width=0.65\textwidth]{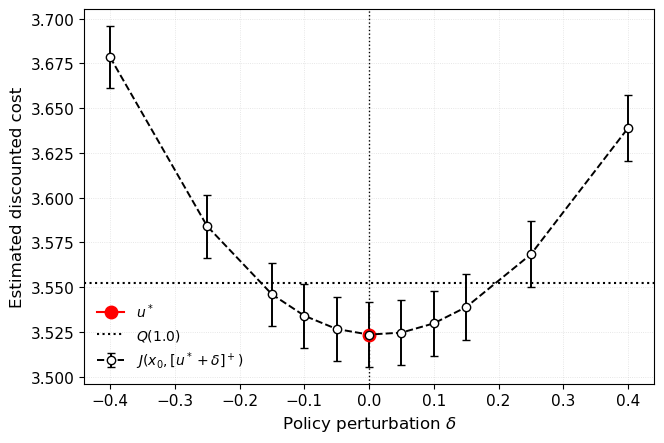}
\caption{Estimated discounted cost under the admissible perturbed feedback policies $u_\delta(x)=[u^*(x)+\delta]^+$.}
\label{fig:policy_perturbation}
\end{figure}

\begin{table}[htbp]
\centering
\caption{Cost comparison for perturbations of the DCP-optimal feedback policy.}
\label{tab:perturbed_policy_costs}
\begin{tabular}{cccc}
\hline
$\delta$
& $\hat J$
& Control cost
& Idleness cost \\
\hline
$-0.40$ & 3.6785  & 2.3955 & 1.2829 \\
$-0.25$ & 3.5840  & 2.1038 & 1.4802 \\
$-0.15$ & 3.5460  & 1.9275 & 1.6185 \\
$-0.10$ & 3.5341 & 1.8445 & 1.6896 \\
$-0.05$ & 3.5266  & 1.7647 & 1.7619 \\
$\mathbf{0.00}$ & $\mathbf{3.5235}$ & 1.6881 & 1.8355 \\
$0.05$  & 3.5246  & 1.6145 & 1.9101 \\
$0.10$  & 3.5298  & 1.5439 & 1.9859 \\
$0.15$  & 3.5389  & 1.4761 & 2.0628 \\
$0.25$  & 3.5684  & 1.3488 & 2.2196 \\
$0.40$  & 3.6389  & 1.1768 & 2.4621 \\
\hline
\end{tabular}
\end{table}

Table~\ref{tab:perturbed_policy_costs} and Figure~\ref{fig:policy_perturbation} show that the smallest estimated total cost among the perturbations considered is attained at $\delta=0$, corresponding to the DCP-optimal feedback $u^*$. 
As $\delta$ increases, the control-cost component decreases while the idleness-cost component increases, illustrating the trade-off between the two components of the objective. 
Perturbations away from the HJB feedback generally increase the estimated discounted cost. 
It provides a more direct numerical check of the local performance of the HJB feedback. These simulations are consistent with the analytical verification result in Theorem~\ref{An optimal control}. 
Thus, the numerical results illustrate both the cost advantage of the HJB feedback relative to the selected constant policies and its robustness relative to nearby admissible perturbations.

\end{appendices}


\bibliography{sn-bibliography}% common bib file


\end{document}